\documentclass[final,12pt]{elsarticle}

\usepackage{amssymb}

\usepackage{braket,amsfonts}
\usepackage{mathtools}
\usepackage[amsmath]{ntheorem}
\usepackage{array}
\usepackage{multirow}
\usepackage{enumitem}
\usepackage{xspace}
\usepackage{arydshln}
\usepackage{lineno}
\usepackage{caption}
\usepackage{subcaption}
\usepackage{algcompatible}
\usepackage[linesnumbered,lined]{algorithm2e}
\RestyleAlgo{ruled}
\usepackage{hyperref}
\usepackage[capitalise,nameinlink,sort]{cleveref}

\usepackage{xcolor}
\usepackage{bm}
\newcommand{\norm}[1]{\lVert#1\rVert}
\newcommand{\R}{\mathbb R}
\newcommand{\N}{\mathbb N}

\newcommand{\bp}{\mathbf p}
\newcommand{\br}{\mathbf r}
\newcommand{\bs}{\mathbf s}
\newcommand{\bu}{\mathbf u}
\newcommand{\bv}{\mathbf v}

\newcommand{\calo}{\mathcal O}

\newcommand{\wh}{\widehat}
\newcommand{\wt}{\widetilde}

\DeclareMathOperator*{\argmax}{\arg\!\max}

\DeclareMathSymbol{\mrq}{\mathord}{operators}{`'}
\DeclareMathSymbol{\mlq}{\mathord}{operators}{``}

\newcommand{\BlackBox}{\rule{1.5ex}{1.5ex}}  
\ifdefined\proof
    \renewenvironment{proof}{\par\noindent{\bf Proof\ }}{\hfill\BlackBox\\[2mm]}
\else
    \newenvironment{proof}{\par\noindent{\bf Proof\ }}{\hfill\BlackBox\\[2mm]}
\fi
\newtheorem{theorem}{Theorem}[section]
\newtheorem{example}[theorem]{Example}
\newtheorem{proposition}[theorem]{Proposition}

\newtheorem{experiment}[theorem]{Experiment}

\crefname{experiment}{Experiment}{Experiments}
\crefname{example}{Example}{Examples}
\crefname{proposition}{Proposition}{Propositions}
\crefname{definition}{Definition}{Definitions}
\crefname{theorem}{Theorem}{Theorems}
\crefname{lemma}{Lemma}{Lemmas}
\usepackage{pgfplots}
 \pgfplotsset{compat=newest,legend style={font=\scriptsize,row sep=-0.05cm,/tikz/every odd column/.append style={column sep=0.01cm}}} 
 \usepgfplotslibrary{groupplots}
 \newcommand{\exponent}{2}
\newlength\figureheight
\newlength\figurewidth
\newcommand{\smtxa}[2]{
{\mbox{\scriptsize
$\left[\!\!
\begin{array}{#1}
#2
\end{array} \!\! \right]$}}}

\usepackage{color}
\definecolor{orange}{rgb}{1.0,0.4,0}

\journal{xxxxx}

\begin{document}

\begin{frontmatter}


\author{Perfect~Y.~Gidisu\corref{cor1}}
\ead{p.gidisu@tue.nl}
\cortext[cor1]{Corresponding author}

\title{Block Discrete Empirical Interpolation Methods\tnoteref{label1}}

\author{Michiel E.~Hochstenbach}
\ead{m.e.hochstenbach@tue.nl}

\begin{abstract}%
We present block variants of the discrete empirical interpolation method (DEIM); as a particular application, we will consider a CUR factorization. The block DEIM algorithms are based on the concept of the maximum volume of submatrices and a rank-revealing QR factorization. We also present a version of the block DEIM procedures, which allows for adaptive choice of block size. The results of the experiments indicate that the block DEIM algorithms exhibit comparable accuracy for low-rank matrix approximation compared to the standard DEIM procedure. However, the block DEIM algorithms also demonstrate potential computational advantages, showcasing increased efficiency in terms of computational time.
\end{abstract}

\end{frontmatter}

\section{Introduction}\label{sec:intro}
A CUR decomposition approximates a data matrix using a subset of its rows and columns. Such factorization preserves in the reduced matrix, properties such as interpretability, sparsity, and nonnegativity of the original data matrix. In machine learning, one can use a CUR decomposition as an unsupervised feature or sample selection technique. Given an $m\times n$ matrix $A$ and a target rank $k$, a CUR factorization takes the form (In line with \cite{Str19}, we will use the letter $M$ rather than $U$ for the middle matrix)
\begin{equation*}
\begin{array}{ccccc}
A&\approx & C &M& R~, \\
 m\times n& & m\times k &k\times k&k\times n
\end{array}
\end{equation*}
where $C$ and $R$ (both of full rank) are small subsets of the columns and rows of $A$, respectively. The selected columns and rows tend to capture the most important information of $A$.  {The methods employed for constructing a CUR decomposition can be broadly classified into two groups: randomized and deterministic algorithms. Randomized methods often exhibit lower computational complexity than their deterministic counterparts, making them more suitable for handling large-scale matrices. However, these methods typically require oversampling of both columns and rows beyond the specified target rank $k$ to achieve strong provable approximation guarantees \cite{optimalboutsidis,deshpande2006matrix,Drineas,frieze2004fast,guruswami2012optimal}. Consequently, they may not be the preferred choice when fixing a rank parameter $k$ and choosing exactly $k$ columns and rows is desired.  

In this paper, our discussions will be limited to deterministic algorithms, which utilize the singular value decomposition (SVD) or rank-revealing QR factorizations  \cite{bischof1991structure,chandrasekaran1994,gu1996,Voronin} and the selection of exactly $k$ columns and rows.}

One standard deterministic approach for constructing the factors $C$ and $R$ of a CUR factorization is to apply the discrete empirical interpolation index selection method (DEIM) to the $k$ dominant right and left singular vectors of $A$, respectively \cite{Sorensen}. Given the factors $C$ and $R$, following \citeauthor{Sorensen} \cite{Sorensen} and others \citep{Mahoney,Stewart}, the full-rank middle matrix $M$ can be computed as $(C^T C)^{-1}C^T A R^T(R R^T )^{-1}$. 

In this work, we propose different block DEIM procedures. There are three motivating factors for a block DEIM scheme:
\begin{itemize}
\item The DEIM algorithm may be considered a {\em greedy} algorithm for finding a submatrix with the maximum absolute determinant in a thin-tall matrix: it locally selects the index corresponding to the largest magnitude element of a vector.
A block DEIM method shares the same principle but may be less greedy since the optimization is done over more indices instead of just one.
\item The proposed block DEIM procedures will generally have higher flop counts than the standard DEIM algorithm (see \cref{tab:timeoverview}). However, in practice, we expect the block DEIM schemes to have higher flop performance than the classical DEIM algorithm as they are mainly based on level 3 {\sf BLAS} building blocks, which perform matrix-matrix operations.
\item A block DEIM scheme may be a good solution where the DEIM procedure faces a difficult choice when the local maximizer is (nearly) nonunique. Given the basis vectors $\bv_i$ for $i=1, \dots, k$, indeed, in \cite[Footnote~3]{Sorensen} it is already hinted that a potentially problematic case
is if multiple entries in a vector being considered have nearly the same magnitude, e.g., $|(\bv_1)_\ell| \approx |(\bv_1)_j|$ for $\ell \ne j$, then the DEIM scheme may sometimes make a relatively arbitrary choice. However, in cases where the nonselected large-magnitude entries in $\bv_1$ are not influential in the subsequent $\bv_i$ vectors, their corresponding indices may never be picked, despite their equal importance; we will give an illustrative case in \cref{example1}).
\end{itemize}

Our selection criterion for the block-$b$ case will consist of maximizing the modulus of the determinant (volume) over many possible $b \times b$ submatrices. The choice of a typical block size depends largely on the size of the input matrix. Given $U\in\R^{m\times k}$,  if $m$ and $k$ are relatively small, a smaller block size may suffice. For larger matrices, a larger block size can be more efficient. One can choose any value for $b$ as long as it is within the range $2\le b < k$. We also discuss the option of taking adaptive values of $b$. 

Some main properties and time complexities of the methods are summarized in \cref{tab:overview,tab:timeoverview}. Since we use the {\sf MaxVol} algorithm (see \cref{sec:review})  and column-pivoted QR decomposition in the block DEIM variants, the procedures have $\calo((m+n)bk)$ more floating point operations than the standard DEIM.  

\begin{table}[htb!] 
\footnotesize \centering
\caption{Overview of the various DEIM variants with their properties. 
All methods require the (approximate) SVD as an input. The speed provides a general indication of performance potential and may vary depending on, for instance, the programming language. \label{tab:overview}}
\scriptsize
\begin{tabular}{l|ccccccc} \hline \rule{0pt}{3ex}%
Method \ $\backslash$ Properties & SVD & Block & Adapt. & Tunable & Greediness & Speed \\[0.5mm] \hline \rule{0pt}{2.3ex}%
Standard DEIM & $+$ &$-$ & $-$ &$-$& High  &Moderate \\
B-DEIM-MaxVol &$+$& $+$ & $-$ &$+$ & Low  & Fast \\
B-DEIM-RRQR &$+$ & $+$ & $-$ &$+$ & Low  & Fast \\
AdapBlock-DEIM &$+$ & $+$ & $+$ &$+$ & Medium  & Variable \\ \hline
\end{tabular}

\end{table}

\begin{table}[htb!]
\footnotesize \centering 
\caption{Summary of the dominant work (time complexity) of the various algorithms for selecting $k$ columns and rows of an $m\times n$ matrix. The notation $\kappa$ denotes the number of iterations needed in the {\sf MaxVol} algorithm.
\label{tab:timeoverview}}
\footnotesize
\begin{tabular}{l|c}\hline \rule{0pt}{2.5ex}%
Method     &  Time complexity\\[0.5mm] \hline \rule{0pt}{2.3ex}%
DEIM     & $\calo(m+n)k^2$\\
{\sf MaxVol}& $\calo(\kappa(m+n)k^2)$\\
B-DEIM-RRQR&$\calo((m+n)(bk +k^2))$\\
B-DEIM-MaxVol&$\calo((m+n)(\kappa bk +k^2))$\\ \hline
\end{tabular}
\end{table}

The outline of the paper is as follows. \Cref{sec:review} provides a brief literature review of some existing deterministic CUR approximation algorithms. \Cref{sec:block} describes how the DEIM scheme can be blocked using either a column-pivoted QR factorization or the concept of the maximum 
 absolute determinant (volume) of submatrices. We discuss the newly proposed block DEIM algorithms and their computational complexities and state a well-known error bound for the approximation. \Cref{sec:exper} reports the results from numerical experiments evaluating the computational efficiency and approximation quality of the various block DEIM procedures. \Cref{sec:con} summarizes the key points and results of this paper.

\section{Review of CUR factorization algorithms}\label{sec:review}
In this section, we summarize some known index selection algorithms for computing a CUR factorization. We denote the spectral norm (2-norm) by $\norm{\cdot}$ and the notation $R^+$ denotes the Moore--Penrose pseudoinverse of $R$. We use MATLAB notations to index vectors and matrices, i.e., $A(:,\bp)$ denotes the $k$ columns of $A$ whose corresponding indices are in vector $\bp \in \N_+^k$. 

\subsection{{\sf MaxVol} Algorithm}

Maximum-volume submatrices play an important role in the construction of a CUR factorization of a matrix \cite{goreinov1997, Goreinov2010}. In this context, the volume of a matrix is defined as the absolute value of its determinant. While computing a submatrix of exact maximal volume is known to be an NP-hard problem \cite{Bartholdi}, in many practical applications, obtaining a submatrix of sufficiently large volume is often adequate. This can be efficiently achieved in polynomial time using the {\sf MaxVol} algorithm \cite{Goreinov2010}. 

The {\sf MaxVol} scheme is a search method designed to identify the submatrix with the largest volume in a tall-thin matrix. The method involves selecting a subset of rows and columns from the input matrix to create a smaller, well-conditioned submatrix. Operating as a greedy iterative approach, the {\sf MaxVol} algorithm systematically swaps rows to maximize the volume of a square submatrix.

First introduced by \citeauthor{goreinov1997} \cite{goreinov1997}, the {\sf MaxVol} method played a pivotal role in constructing a rank-$k$ CUR factorization known as the pseudoskeleton approximation. Subsequently, \citeauthor{oseledets2010tt} \cite{oseledets2010tt} developed a cross-approximation scheme that alternates between selecting rows and columns using the {\sf MaxVol} algorithm. In a recent study, \citeauthor{Mikhalev} \cite{Mikhalev} extended the results of square maximum-volume submatrices to the rectangular case and proposed an algorithm for constructing a rectangular submatrix with maximum volume.

Given a tall-thin matrix $U \in \R^{m\times k}$ containing the $k$-leading singular vectors, the {\sf MaxVol} procedure searches for $k$ row indices such that the resulting $k\times k$ upper submatrix $\wh U$ is dominant in $U$ \cite{Goreinov2010}. This means that  $|U\wh U^{-1}|_{ii}\le1$. While the dominant property does not necessarily imply that $\wh U$ has the maximum volume, it does guarantee that $\wh U$ is locally optimal. This implies that replacing any row in $\wh U$ with a row from $U$ that is not already present in $\wh U$ will not increase the volume. The procedure for finding a dominant submatrix using the {\sf MaxVol} algorithm is outlined in \cref{algo:maxvol}.

\begin{algorithm}[htb!]
{\footnotesize
\KwData {$U \in \R^{m \times k}$ with $m> k$, convergence tolerance $\delta$ (default 0.01)}

\KwResult {$\bs \in \N_+^k$ indices}

$\bs \leftarrow$ indices of the first $k$ pivoted rows from  LU decomposition of $U$

\Repeat{{ $\forall$ $(i,j)$: $|b_{ij}|<1+\delta $ }\label{algo:stopline}}{

Set $\wh U\leftarrow U(\bs,:)$ and $B\leftarrow U\wh U^{-1}$ \label{algo:invline}

Find the element of maximum absolute value in $B$: $(i, j)\leftarrow \argmax |b_{ij}|$

if $|b_{ij}|>1$, swap rows $i$ and $j$ in $B$: \ $\bs(j) = i$ \label{algo:swapline}
}
 \caption{{\sf MaxVol}: Approximation to dominant submatrix \cite{Goreinov2010}}\label{algo:maxvol}}
\end{algorithm}
In practice, a useful initialization step for the {\sf MaxVol} algorithm is to use the pivoted rows from the LU decomposition of the input matrix as the starting point \cite{Goreinov2010}.
The parameter $\delta$ is the convergence tolerance to find pivot elements; this parameter serves as a stopping criterion and should be sufficiently small (a good choice can be $0.01$) \cite{Goreinov2010}. Note that by swapping the rows as done in \cref{algo:swapline}, the volume of the upper submatrix in $B$ is increased, and also in $U$ until convergence. The most expensive part of the iterations is \cref{algo:invline}; this needs a $k\times k$ matrix inversion and $\calo(mk^2)$ operations for the matrix multiplication. Goreinov et al. \cite{Goreinov2010} describe a speed optimization process that avoids the expensive matrix multiplications and inversions. We refer the reader to \cite{Goreinov2010} for a more detailed explanation of the {\sf MaxVol} approach. 

With regards to computational cost, the initialization step of \Cref{algo:maxvol} requires a permutation of the input matrix, which can be done via the LU factorization. Given an $m \times k$ matrix, the LU decomposition requires $\calo(mk^2)$ operations. Furthermore, the dominant cost of each iteration in the algorithm is the multiplication of an $m\times k$ matrix and a $k \times k$ matrix, for a cost of $\calo(mk^2)$ operations. Let $\kappa \in \N_+$ denote the number of iterations performed. The computational complexity of the {\sf MaxVol} procedure is $\calo(\kappa mk^2)$. A crude bound on the number of iterations in \cref{algo:maxvol} is $\kappa \le (\log|{\sf det}(\wh U_{\sf dom})|-\log|{\sf det}(\wh U_{\sf ini})|)/\log(1+\delta)$, where $\wh U_{\sf ini}$ is the submatrix at the initialization step and $\wh U_{\sf dom}$ is the dominant submatrix in \cref{algo:maxvol} \cite{Goreinov2010}.

\subsection{Discrete Empirical Interpolation Method}
The DEIM point selection method is a deterministic greedy index selection algorithm originally presented in the context of model order reduction for nonlinear dynamical systems \cite{Barrault,Chaturantabut}. \citeauthor{Sorensen} \cite{Sorensen} show that this procedure is a viable index selection algorithm for constructing a CUR factorization. To select the indices for a rank-$k$ CUR factorization, first, compute a rank-$k$ SVD of the original matrix. Using as input the top-$k$ right and left singular vectors contained in $V$ and $U$, the DEIM algorithm selects $k$ column and row indices, denoted as $\bp$ and $\bs$, respectively, as in \cref{algo: DEIM}. \footnote{Note that the backslash operator used in the algorithms is a Matlab-type notation for solving linear systems and least-squares problems.} The DEIM procedure selects the indices by processing the singular vectors one at a time. We describe this method using the left singular vectors, and the procedure on the right singular vectors follows similarly. 

Starting from the leading left singular vector $\bu_1$, the first row index $s_1$ corresponds to the entry in $\bu_1$ with the largest magnitude, i.e., $|\bu_1(s_1)|=\norm{\bu_1}_\infty$ (where $\norm{\cdot}_\infty$ denotes the infinity norm). The remaining indices $s_j$ for $j=2, \dots, k$ are selected so that each index corresponds to the largest magnitude entry in the residual $\br_j=\bu_j-\mathbb{S}_{j-1}\bu_j$, where $\mathbb{S}_{j-1}$ is an interpolatory projector computed as $U_{j-1}(S_{j-1}^TU_{j-1})^{-1}S_{j-1}^T$ with $S$ being an identity matrix indexed by the $j-1$ selected indices. The linear independence of the columns of $U$ guarantees that $S_{j-1}^TU_{j-1}$ is nonsingular. As mentioned in our third motivating factor in \cref{sec:intro}, in the case where we have multiple index options, e.g., $|(\br_j)_1| = |(\br_j)_2|$, the smaller index is picked. For further details about the DEIM scheme, we refer the reader to \citep{Chaturantabut,Sorensen}.

\begin{algorithm}[htb!]
{\footnotesize
\KwData {$U \in \R^{m \times k}$  with $k\le m$ (full rank)}

\KwResult {Indices $\bs \in \N_+^k$ with non-repeating entries} 

$\bs(1)$ =  $\argmax_{1\le i\le m}~ |(U(:,\,1))_i|$

\For{ $j = 2, \dots, k$}{

$U(:,\,j) = U(:,\,j)-U(:,\,1:j-1)\cdot (U(\bs,\,1:j-1)
\ \backslash \ U(\bs,\,j))$  

$\bs(j)$ =  $\argmax_{1\le i\le m}~ |(U(:,\,j))_i|$\hspace{3mm}
}
  \caption{Discrete empirical interpolation index selection method \cite{Chaturantabut}}\label{algo: DEIM}}
\end{algorithm}

It is noteworthy to highlight a modification of the DEIM scheme proposed in \cite{Drmac}. This variant, known as QDEIM, employs a column-pivoted QR factorization on the transpose of matrices $U$ and $V$ to select the indices. The QDEIM algorithm offers a simpler approach compared to the original DEIM, yet maintains a projection error bound within the same order of magnitude. The algorithm's efficiency is further underscored by the prevalence of efficient column-pivoted QR implementations in various open-source packages, making it a practical and effective alternative for index selection in numerical applications.

\subsection{Pivoted QR factorization}
The classical truncated pivoted QR factorization is another approach for computing a CUR factorization \citep{Voronin}. A column-pivoted (rank-revealing) QR factorization (RRQR) of a matrix $A\in \R^{m\times n}$ with $m\ge n$ is of the form 

\begin{linenomath}\begin{equation*}
\begin{array}{ccccc}
A&\Pi & = &Q& T, \\
m\times n & n\times n &&m\times n &n\times n
\end{array}
\end{equation*}\end{linenomath}
where $\Pi$ is a permutation matrix, $Q$ is a matrix with orthonormal columns, and $T$ is an upper triangular matrix that satisfies the condition \cite{Voronin}
\[|T_{kk}|^2\ge \sum_{i=k}^j|T_{ij}|^2, ~j=k+1, \dots, n, \quad k=1, \dots, n.\] 

The QR factorization is commonly constructed incrementally using a greedy algorithm such as the column-pivoted Gram-Schmidt method. This approach allows for the option of stopping the process after the computation of the first $k$ terms, resulting in a partial QR factorization of $A$.
Given $\bp$, a vector of indices, we can express $\Pi$ as $I(:,\bp)$. Suppose we partition $Q$ and $T$ so that

\begin{equation}\label{eq:pivotedqr}
A(:,\bp) = [Q_1\ \ Q_2]\begin{bmatrix}
 {T_{11}} & {T_{12}} \\
 0 & T_{22} 
\end{bmatrix}=Q_1\,[T_{11}\ \ T_{12}] + Q_2\,[0 \ \ T_{22}],
\end{equation}
where $Q_1\in \R^{m\times k}$, $Q_2\in \R^{m\times(n-k)}$, $T_{11} \in \R^{k \times k}$, $T_{12} \in \R^{k\times (n-k)}$, \\$T_{22}\in \R^{(n-k)\times (n-k)}$, and $\wh A_k:=Q_1\,[T_{11}\ \ T_{12}]$, we have 

\[\norm{A\Pi- \wh A_k}\le\norm{T_{22}}\]
to be the error bound of a truncated pivoted QR decomposition of $A$. 
This implies that $Q_1$ is an approximation of the range of $A$ and as long as $\norm{T_{22}}$ is small, $A(:,\bp)$ can be approximated by $\wh A_k$. For an arbitrary $k$, the best rank-$k$ approximation of $A$ $(A_k)$ from the SVD gives $\norm{A-A_k}=\sigma_{k+1}(A)$, where $\sigma_{k+1}$ denotes the $(k+1)$st singular value. It is always the case that $\sigma_{k+1}(A)\le \norm{T_{22}}$.

The traditional column-pivoted Gram-Schmidt algorithm usually leads to an RRQR (i.e., $\norm{T_{22}}$ is small or close to $\sigma_{k+1}$). Nonetheless, there are instances where this method falls short of generating a factorization that yields a small value of $\norm{T_{22}}$ (see, e.g., \cite{gu1996}). There are several ways to compute an RRQR factorization \cite{chan1987rank,chandrasekaran1994,matrixcomp,gu1996}. The
computational complexities of these methods are slightly larger than the standard QR decomposition algorithm (see \cite{boutsidis2008selecting} for a tabulated comparison of various RRQR schemes).

A rank-revealing QR factorization is based on selecting certain well-conditioned submatrices \citep{chandrasekaran1994}, which is the underlying objective of the maximum volume concept. Notice that from the matrix partition, $Q_1T_{11}$ equals the first $k$ columns of $A(:,\bp)$. To construct a CUR decomposition, one may apply a Gram--Schmidt-based pivoted QR algorithm to the matrices $A$ and $A^T$ to obtain the matrices $C$ and $R$, respectively \citep{berry2005,Stewart}. Alternatively, \cite{Voronin} show how to compute a CUR factorization via a two-sided interpolative decomposition (ID), which in turn can be constructed from a column-pivoted QR factorization. Using the QDEIM approach \citep{Drmac}, one can apply a column-pivoted QR procedure on the transpose of the leading $k$ right and left singular vectors to find the indices for constructing the factors $C$ and $R$ in CUR approximation.

We will now use the tools described in this section to design our block variants of the standard DEIM algorithm in the following section.

\section{Block DEIM}\label{sec:block} This section introduces new block variants of the DEIM procedure. We combine the index selection algorithms discussed in \cref{sec:review} to design our block DEIM algorithms. The block DEIM algorithms select a $b$-size set of indices at each step using a block of singular vectors and then update the subsequent block of vectors using the oblique projection technique in the standard DEIM procedure. The main difference between the block DEIM and the traditional DEIM scheme resides in the $b$-size indices selection. In the standard DEIM, we select the indices by processing one singular vector at a time. Each iteration step produces an index whilst the block version processes a block of singular vectors at a time, and each step provides an index set of size $b$. As a result, this may lead to a different selection of indices. For ease of presentation, we assume that the number of singular vectors $k$ is a multiple of the block size $b$ and $2\le b < k$.

\subsection{Block DEIM based on maximum volume/RRQR} \label{sec:dmaxvol}
As our first block DEIM procedure, we present a block DEIM method that can be constructed using either the MaxVol procedure (B-DEIM-MaxVol) or an RRQR factorization (B-DEIM-RRQR), summarized in \cref{algo:B-Maxvol}.
For the standard DEIM method, each next index is picked greedily based on the maximal absolute value of an oblique projected singular vector. For this block variant, in every step, we greedily pick a fixed number ($b$) of indices based on an (approximate) maximal volume (absolute value of determinant) of a $b \times b$ submatrix of the projected singular vectors. To this end, we first exploit the {\sf MaxVol} scheme \cite{Goreinov2010} described in \cref{sec:review} to efficiently find a submatrix with the approximately maximal determinant (with a given tolerance $\delta$) in a thin-tall matrix. 

The {\sf MaxVol} algorithm seeks to find a good approximation of a submatrix with the maximum volume in a given input matrix.  It is important to note that this {\sf MaxVol} scheme and an RRQR decomposition share a common principle, i.e., finding a well-conditioned submatrix. As noted in \cite[Remark~2.3]{Drmac}, the pivoting in an RRQR factorization can be interpreted as a greedy volume maximizing scheme. Practical experience shows that the index selection via a column-pivoted QR factorization may be more computationally efficient than the {\sf MaxVol} method. With this insight in mind, we propose an alternative way of computing the block DEIM by employing a column-pivoted QR factorization. The B-DEIM-RRQR may be viewed as a hybrid standard DEIM and QDEIM scheme \cite{Drmac}. The QDEIM selects $k$ indices by performing one column-pivoted QR while our B-DEIM-RRQR algorithm selects the $k$ indices by performing $k/b$ rounds of a column-pivoted QR. Note that when the block size $b=k$, this B-DEIM-RRQR algorithm is just the QDEIM scheme.

\begin{algorithm}[htb!]
{\footnotesize
\KwData {$U \in \R^{m \times k}$, $V \in \R^{n \times k}$, $k\le \min(m,n)$, block size $b$ (with $b\,|\,k$), convergence tolerance $\delta$ (default 0.01)}

\KwResult {Indices $\bp, \bs \in \N_+^k$ with non-repeating entries} 

\For{ $j = 1, \dots, k/b$}{
\textbf{if} method = {\sf MaxVol}

\hskip1.5em $\bs((j-1)b+1:jb)={\sf MaxVol}(U(:,(j-1) b+1:jb),\,\delta)$ \label{algo:l1}

\hskip1.5em $\bp((j-1)b+1:jb)={\sf MaxVol}(V(:,(j-1)b+1:jb),\,\delta)$ \label{algo:l2}

\textbf{else}

\hskip1.5em Perform a column-pivoted QR on $U(\,:,(j-1) b+1:jb)^T$ and 

\hskip1.5em $V(\,:,(j-1) b+1:jb)^T$\label{step1}, giving the permutations $\Pi_\bs$ and $\Pi_\bp$ 

\hskip1.5em $\bs((j-1) b+1:jb)=\Pi_\bs(1:b)$ \label{perm1}

\hskip1.5em $\bp((j-1) b+1:jb)=\Pi_\bp(1:b)$ \label{perm2}

\textbf{end if}

Let ${\rm cols}=jb+1:jb+b$

$U(:,{\rm cols}) = U(:,{\rm cols}) - U(:,1:jb)\cdot(U(\bs,\,1:jb)\ \backslash \ U(\bs,\,{\rm cols}))$ \label{algo:l4}

$V(:,{\rm cols}) = V(:,{\sf col}) - V(:,1:jb)\cdot(V(\bp,\,1:jb)\ \backslash \ V(\bp,\,{\rm cols}))$ \label{algo:ll5}
}
 \caption{Block DEIM index selection based on {\sf MaxVol}}\label{algo:B-Maxvol}}
\end{algorithm}

In \cref{algo:l4,algo:ll5} of \cref{algo:B-Maxvol}, we use the oblique projection technique as in the standard DEIM scheme to update the subsequent blocks of singular vectors. 

We will now describe how \cref{algo:B-Maxvol} selects the row indices; the selection of the column indices follows similarly. \cref{algo:B-Maxvol} starts from the leading-$b$ dominant left singular vectors $U_{b}$. The initial index vector set, denoted as $\bs_{b}$, consists of the first $b$ row indices corresponding to the dominant submatrix in $U_{b}$. This submatrix can be obtained by either applying \cref{algo:maxvol} to $U_{b}$ or by performing a QR decomposition with column pivoting, should you opt for the B-DEIM-RRQR variant.  Let $\bs=[\bs_{b}]$, $S_{b}=I(:,\bs_{b})$, and define an oblique projection operator as $\mathbb{S}_{b}=U_{b}(S_{b}^TU_{b})^{-1}S_{b}^T$. 

Suppose we have $(j-1)b$ indices, so that 
\[\bs_{(j-1)b}=\smtxa{c}{s_{1}\\\vdots\\s_{(j-1)b}}, \quad S_{(j-1)b}=I(:,\bs_{(j-1)b}), \quad U_{(j-1)b}=[\bu_{1}, \dots, \bu_{(j-1)b}],\]
and 
\[\mathbb{S}_{(j-1)b}=U_{(j-1)b}\,(S_{(j-1)b}^TU_{(j-1)b})^{-1}\,S_{(j-1)b}^T.\]
Let $U_{jb}=U(:,jb+1\!:\!jb+b)$. Compute the updated vectors (residual) $U_{jb}=U_{jb}-\mathbb{S}_{(j-1)b}U_{jb}$ (see \cref{algo:l4} of the algorithm), and then select the subsequent sets of $b$ indices by applying \cref{algo:maxvol} (or a column-pivoted QR) to the updated matrix $U_{jb}$. It is worth noting that, using this oblique projection operator $\mathbb{S}_{(j-1)b}$ on the original $U_{jb}$ ensures that the $\bs_{(j-1)b}$ entries in the updated $U_{jb}$ are zero, which guarantees nonrepeating indices. At the end of the iteration, the algorithm returns a column and row index set of size $k$.

The following are two potential benefits of \cref{algo:B-Maxvol} compared to the standard DEIM.
\begin{itemize}
\item Approximation-wise, the greedy selection is not column-by-column, but carried out on a block of columns, possibly leading to a better pick of indices. Since the value of $b$ is usually modest, this procedure is still very affordable.
\item Computationally, the block procedure may have some benefits over the vector-variant, for instance in the work for the oblique projection.
\end{itemize}

The computational cost of the B-DEIM-Maxvol variant of this algorithm is primarily determined by two key factors. Firstly, the two calls of the MaxVol procedure contribute to the overall complexity, where the {\sf MaxVol} procedure has a time complexity of $\calo(\kappa(m+n)b^2)$. Secondly, the block updates (as described in \cref{algo:l4,algo:ll5}) contribute with complexities of  $\calo((m+n)k^2)$. Considering that \cref{algo:B-Maxvol} comprises $k/b$ iterations, the overall computational cost of the B-DEIM-MaxVol algorithm is $\calo((m+n)(\kappa bk +k^2))$.

On the other hand, the cost of the B-DEIM-RRQR scheme is dominated by the two QR factorizations as well as the block updates. Given an $n \times b$ matrix, a QR factorization necessitates $\calo(nb^2)$ operations. Consequently, the combined cost of the two QR decompositions in this scheme amounts to $\calo((m+n)b^2)$. Thus, the total computational cost of the B-DEIM-RRQR scheme after $k/b$ iterations is $\calo((m+n)(kb +k^2))$.

\subsection{Adaptive block DEIM} \label{sec:dadap}
In this section, we consider adaptive choices for the block size $b$. 
In particular, we propose the following variant called AdapBlock-DEIM:
Perform block DEIM if for the singular vector $\bv_j$ being considered, the two largest elements are (nearly) equal (see \cref{example1}), i.e., 
\[
\left\{ 
\begin{array}{ll}
{\rm block~DEIM} & \quad {\rm if} \ |(\bv_j)_i| \approx |(\bv_j)_\ell| \quad {\rm for}\ i \ne \ell, \\[1mm]
{\rm standard~DEIM} & \quad \text{otherwise}.
\end{array}
\right.
\]

As outlined in the \cref{sec:intro}, there are three motivating factors for the block DEIM variants. The adaptive DEIM aim to address the final motivation, specifically, when the DEIM procedure faces a difficult choice due to the (near) non-uniqueness of local maximizers. The adaptive DEIM can be applied selectively, intended for situations where the use of a block DEIM scheme is desired only when multiple entries exhibit nearly identical magnitudes. This consideration does not necessarily suggest any shortcomings of block DEIM; rather, the adaptive DEIM represents an approach that closely aligns with the original DEIM under specific conditions.

\begin{algorithm}[htb!]
{\footnotesize
\KwData {$U \in \R^{m \times k}$, $V \in \R^{n \times k}$, $k\le \min(m,n)$, block size $b$, $\rho$ (default 0.95), tolerance $\delta$ (default 0.01), method ({\sf MaxVol} or {\sf QR})}

\KwResult {Indices $\bp, \bs \in \N_+^k$ with non-repeating entries} 

$j=1$

\While{$j\le k$}{
\textbf{if} $j>1$ \textbf{then}

\hskip1.5em $\wt \bu=U(\bs(1:j-1),1:j-1) \ \backslash \ U(\bs(1:j-1),j)$

\hskip1.5em $U(:,j) = U(:,j) - U(:,1:j-1)\cdot\wt \bu$

\textbf{end if}

Let $u_1$ and $u_2$ be the two largest entries of $U(:, j)$ in magnitude

Let ${\sf ind}(1)$ be the index corresponding to $u_1$

\textbf{if} $(j+b-1>k) \ {\bf or} \ (u_2 < \rho \cdot u_1$) \textbf{then} $\bs(j)={\sf ind}(1)$; \ $j=j+1$

\textbf{else} 

\hskip1.5em ${\rm cols}=j+b-1$

\hskip1.5em \textbf{if} $j>1$ \textbf{then}

\hskip2.5em $\widetilde U= U(:,j+1:{\rm cols})$; \ \ $\widehat U= U(\bs(1:j-1),j+1:{\rm cols}))$

\hskip2.5em $\widetilde U = \widetilde U- U(:,1:j-1)\cdot(U(\bs(1:j-1),1:j-1) \ \backslash \ \widehat U$

\hskip2.5em $U(:,j+1:{\rm cols})=\widetilde U$

\hskip1.5em \textbf{end if}

\hskip1.5em \textbf{if} method = {\sf MaxVol}

\hskip2.5em $\bs(j:{\rm cols})={\sf MaxVol}(U(:, j:{\rm cols}), \, \delta)$ \label{algo:l5} 

\hskip1.5em \textbf{else}

\hskip2.5em Perform a column-pivoted QR on $U(:, j:{\rm cols})^T$\label{step2}, giving

\hskip2.7em permutation $\Pi_\bs$

\hskip2.5em $\bs(j:{\rm cols})=\Pi_\bs(1:b)$

\hskip1.5em\textbf{end if}

\hskip1.5em $j=j+b$

\textbf{end if}
}

Repeat the procedure on $V$ to get indices $\bp$

 \caption{Adaptive block DEIM index selection}\label{algo:Adapblock}}
\end{algorithm}

In \cref{algo:Adapblock}, we show an implementation of the adaptive block DEIM using the block DEIM variants discussed in \cref{sec:block}. The parameter $\rho$ is the desired lower bound on the ratio $|(\bv_j)_i| \, / \, |(\bv_j)_\ell|$ for $i \ne \ell$. Although our criterion for switching from a standard DEIM scheme to a block DEIM method is based on how close the two largest entries (magnitude) in the vector being considered are, other criteria can be used. Furthermore, in Algorithm \ref{algo:Adapblock}, the block size is predetermined before the index selection process. An alternative approach is to adapt the algorithm to dynamically determine a varying block size during the execution of the index selection. Consequently, the block size would correspond to the number of entries in the projected singular vector under consideration that are (nearly) equal given the threshold parameter $\rho$.

\subsection{Error bounds} \label{sec:errbound}
The following proposition restates a known theoretical error bound for a CUR approximation, which holds for the block DEIM algorithms proposed in this paper. A detailed constructive proof is in \cite{Sorensen}; we provide the necessary details here for the reader’s convenience. Let $P\in \R^{n\times k}$ and $S\in \R^{m\times k}$ be matrices with some columns of the identity indexed by the indices selected by employing any of the block DEIM algorithms. 
\begin{proposition}\label{pp1} \cite[Thm.~4.1]{Sorensen}
Given $A\in \R^{m \times n}$ and a target rank $k$, let $U\in \R^{m\times k}$ and $V\in \R^{n\times k}$ contain the leading $k$ left and right singular vectors of $A$, respectively. Suppose $C=AP$ and $R=S^T\!A$ are of full rank, and $V^T\!P$ and $S^TU$ are nonsingular. Then, with $M=C^+\!AR^+$, a rank-$k$ CUR decomposition constructed by either of the block DEIM schemes presented in this paper satisfies
\[\norm{A-CMR}\le(\eta_\bs + \eta_\bp)\,\sigma_{k+1} \quad with \quad \eta_\bs < \sqrt{\tfrac{nk}{3}}\,2^k~, \quad \eta_\bp < \sqrt{\tfrac{mk}{3}}\,2^k,\]
where $\eta_\bp=\norm{(V^T\!P)^{-1}}$, $\eta_\bs=\norm{(S^TU)^{-1}}$.
\end{proposition}

\Cref{pp1} suggests that the error constants $\eta_\bs$ and $\eta_\bp$ are indicators of the quality of an index selection method \cite{Sorensen}.

Let us consider a small illustrative example of potentially improved index selection by the block DEIM method.

\begin{example}\label{example1} {\rm
Consider the matrix of left singular vectors (where, e.g., $\varepsilon=10^{-15}$)
\[
U = \smtxa{cc}{\frac13 \sqrt{3} + {\varepsilon} & \phantom-0 \\[1mm]
\frac13 \sqrt{3} & \phantom-\frac12 \sqrt{2} + \varepsilon \\[1mm]
\frac13 \sqrt{3} & -\frac12 \sqrt{2}}, \quad \text{corresponding to singular values} \ \sigma_1, \ \sigma_2.
\]
It may easily be checked that, independent of the singular values, standard DEIM will
pick the first index, followed by the second.
However, it is clear that the second and third indices are a better choice,
especially if $\sigma_2$ is close to $\sigma_1$.
For instance, if $A = U \cdot \smtxa{cc}{1&0\\[1mm] 0& 0.99}$,
working with a determinant of $2 \times 2$ submatrices gives a better
result. Since
\[
\Big| \det \smtxa{cc}{\frac13 \sqrt{3} & \phantom-\frac12 \sqrt{2} + \varepsilon\\[1mm]
 \frac13 \sqrt{3} & -\frac12 \sqrt{2} } \Big|
> \Big| \det \smtxa{cc}{\frac13 \sqrt{3} + {\varepsilon} & \phantom-0 \\[1mm]
 \frac13 \sqrt{3} & \phantom-\frac12 \sqrt{2} + \varepsilon} \Big|,
\]
a block DEIM variant with block size 2 picks the more appropriate indices $2$ and $3$. Additionally, as noted earlier, the DEIM scheme attempts to minimize the quantity $\norm{(S^TU)^{-1}}$, where $S$ is an index selection matrix. From our simple example above, we have that
\[
\Big\lVert \bigg(\smtxa{cc}{\frac13 \sqrt{3} & \phantom-\frac12 \sqrt{2}+ \varepsilon \\[1mm]
 \frac13 \sqrt{3} & -\frac12 \sqrt{2} }\bigg)^{-1}\Big\rVert
< \Big\lVert \bigg( \smtxa{cc}{\frac13 \sqrt{3} + {\varepsilon} & \phantom-0 \\[1mm]
 \frac13 \sqrt{3} & \phantom-\frac12 \sqrt{2} + \varepsilon}\bigg)^{-1}\Big\rVert.
\]
This implies that the rows selected by the block DEIM schemes may yield a smaller quantity $\eta_\bs = \norm{(S^TU)^{-1}}$ compared to the DEIM procedure.
The chosen indices have a larger determinant and arguably are a better set of indices.
}
\end{example}

\section{Numerical Experiments}\label{sec:exper}
We now conduct several sets of illustrative larger-scale experiments to show the effectiveness of the block DEIM variants, i.e., B-DEIM-MaxVol, B-DEIM-RRQR, 
and AdapBlock-DEIM, proposed in this work using synthetic and real data sets. We evaluate the algorithms by applying them to data analysis problems in several application domains: recommendation system analysis, model order reduction, economic modeling, and optimization. We use real-world sparse and dense data matrices with sizes ranging from small to large scale. An overview is presented in \cref{ch3tab:exps}. We compare the performance of our algorithms for constructing a CUR approximation with three state-of-the-art deterministic algorithms: DEIM \cite{Sorensen}, QDEIM \cite{Drmac}, and {\sf MaxVol} \cite{Goreinov2010}. All these algorithms require the leading $k$ right and left singular vectors to construct a rank-$k$ CUR factorization. We use these two evaluation criteria: the rank-$k$ approximation relative error $\norm{A-CMR}\,/\,\norm{A}$ and the computational efficiency, i.e., the runtime scaling for the rank parameter $k$. Here, the runtime measures the time it takes each algorithm to select the desired number of column and row indices.  We do not consider the run time for computing the singular vectors since all the methods we consider require the SVD. However, it is important to note that the total cost of selecting the indices may be dominated by the computational cost of the SVD (see \cref{tab:lascale} where the runtimes of the SVD are included). Our experiments are not meant to be exhaustive; however, they provide clear evidence that the block DEIM schemes proposed in this paper may provide a comparable low-rank approximation while being computationally more efficient. Note that the runtime serves as a rough indicator of potential performance and may vary depending on, e.g., the programming language.

In the implementation, we perform a column-pivoted QR factorization (as RRQR) and the truncated SVD using the MATLAB built-in functions {\sf qr} and {\sf svds} \cite{baglama2005augmented}, respectively (and {\sf svd} for small cases). For the {\sf MaxVol} algorithm, we use a MATLAB implementation by \citeauthor{Kramer2016} \cite{Kramer2016} made available on GitHub\footnote{\href{https://github.com/bokramer/CURERA/blob/master/maxvol.m}{github.com/bokramer/CURERA/blob/master/maxvol.m}}. Unless otherwise stated, in all the experiments we use as default block size $b=5$ for small/mid-size matrices and $b=10$ for large-scale matrices, the AdapBlock-DEIM method parameter $\rho=0.95$, and the {\sf MaxVol} scheme convergence tolerance $\delta=0.01$. The $U$ and $V$ matrices (containing the left and right singular vectors or their approximates) are in dense format even if A is sparse.

\begin{table}[htb!]
\centering
\caption{Various examples and dimensions considered.\label{ch3tab:exps}}
{\footnotesize
\begin{tabular}{cllcccc} \hline \rule{0pt}{2.3ex}%
{\bf Exp.} & {\bf Domain} & {\bf Matrix} & $m$ & $n$ \\ \hline \rule{0pt}{2.3ex}%
1 & Recommendation system & ~Dense & 14116 & \phantom{12}100 \\
2 & Economic modeling & ~Sparse & 29610 &29610\\
3 & Optimization & ~Sparse & 29920	 &29920	\\
4 & Model order reduction  & ~Sparse & 23412 &23412\\
5 & Structural engineering  & ~Sparse & 22044	 &22044\\
6 & Synthetic & ~Dense & \phantom{1}2000 &\phantom{1}4000\\ \hline
\end{tabular}}
\end{table}

\begin{experiment} {\rm In this first set of experiments, we aim to evaluate how our proposed block-DEIM variants compared with the existing deterministic methods mentioned earlier on a small matrix. Our data set is from the recommendation system analysis domain, where one is usually interested in making service or purchase recommendations to users. One of the most common techniques for recommendation systems is collaborative filtering, which involves recommending to users items that customers with similar preferences liked in the past. The {\sf Jester} data set is often used as a benchmark for recommendation system research \cite{goldberg2001eigentaste}. This data matrix consists of $73421$ users and their ratings for $100$ jokes. We only consider users who have ratings for all $100$ jokes resulting in a $14116 \times 100$ matrix.

\begin{figure}[htb!]
\centering
%
%
\definecolor{mycolor1}{rgb}{0.00000,0.44700,0.74100}%
\definecolor{mycolor2}{rgb}{0.85000,0.32500,0.09800}%
\pgfplotsset{every axis title/.append style={at={(1,1)}}}
{\begin{tikzpicture}

\begin{groupplot}[
group style={
    group name=my plots,
    group size=2 by 2,
    xlabels at=edge bottom,
    ylabels at=edge left,    
    horizontal sep=1.5cm,
    vertical sep=3cm,
    },
     xmin=10,
     xmax=50,
    legend style={at={(-0.2,-0.5)},anchor=south,legend  columns =4},
width=0.5\linewidth,
height=5cm
]

\nextgroupplot[title=MaxVol methods,xlabel=rank-$k$, ylabel= {$\| A - CMR \|\ /\ \| A \|$}]
\addplot [color=red,line width=1.5pt,mark=asterisk]
  table[row sep=crcr]{%
5	0.599490007423766\\
10	0.468121361706147\\
15	0.42454916123905\\
20	0.413765688283733\\
25	0.362264004099457\\
30	0.357194143256654\\
35	0.301780071536928\\
40	0.284994316185363\\
45	0.266970354850498\\
50	0.240610308027583\\
};

\addplot [color=blue,line width=1.5pt,mark=square]
  table[row sep=crcr]{%
5	0.551636602092043\\
10	0.465266084376116\\
15	0.414447205829555\\
20	0.378634825772304\\
25	0.340758052508374\\
30	0.328735562352019\\
35	0.297994959138082\\
40	0.279788416557406\\
45	0.269639296289848\\
50	0.258358190268443\\
};

\addplot [color=green,line width=1.5pt,mark=diamond]
  table[row sep=crcr]{%
5	0.551636602092043\\
10	0.536626790830776\\
15	0.404631363873087\\
20	0.402895695105694\\
25	0.390561505445196\\
30	0.354936244857587\\
35	0.304758098991285\\
40	0.305668640162129\\
45	0.307350010807068\\
50	0.302297385008194\\
};

\addplot [color=black,line width=1.5pt,mark=o]
  table[row sep=crcr]{%
5	0.551636602092043\\
10	0.465266084376116\\
15	0.409154064532598\\
20	0.380594107141777\\
25	0.335979598829517\\
30	0.333617258501672\\
35	0.297347735674794\\
40	0.274447443709111\\
45	0.260615743936223\\
50	0.249624056220473\\
};
\nextgroupplot[xlabel=rank-$k$,scaled y ticks=base 10:\exponent,ylabel=time(s)]

\addplot [color=red,line width=1.5pt,mark=asterisk]
  table[row sep=crcr]{%
5	0.00083102\\
10	0.00232641\\
15	0.00719682\\
20	0.01185963\\
25	0.01873095\\
30	0.02729086\\
35	0.03490451\\
40	0.04343089\\
45	0.05269186\\
50	0.06544395\\
};
\addlegendentry{DEIM}

\addplot [color=blue,line width=1.5pt,mark=square]
  table[row sep=crcr]{%
5	0.00139283\\
10	0.00362625\\
15	0.00553859\\
20	0.0094696\\
25	0.01262166\\
30	0.01529332\\
35	0.01882498\\
40	0.02209524\\
45	0.02908552\\
50	0.03274253\\
};
\addlegendentry{B-DEIM-MaxVol}

\addplot [color=green,line width=1.5pt,mark=diamond]
  table[row sep=crcr]{%
5	0.0010918\\
10	0.01000623\\
15	0.02832728\\
20	0.02361689\\
25	0.04625237\\
30	0.0912537\\
35	0.09682609\\
40	0.10479772\\
45	0.13490289\\
50	0.14543998\\
};
\addlegendentry{MaxVol}

\addplot [color=black,line width=1.5pt,mark=o]
  table[row sep=crcr]{%
5	0.00170038\\
10	0.00534398\\
15	0.01024146\\
20	0.01616331\\
25	0.02131899\\
30	0.03497813\\
35	0.03761877\\
40	0.04526648\\
45	0.05904611\\
50	0.06205741\\
};
\addlegendentry{AdapBlock-MaxVol}
\nextgroupplot[title=RRQR methods, xlabel=rank-$k$,ylabel= {$\| A - CMR \|\ /\ \| A \|$}]

\addplot [color=red,line width=1.5pt,mark=asterisk]
  table[row sep=crcr]{%
5	0.599490007423766\\
10	0.468121361706147\\
15	0.42454916123905\\
20	0.413765688283733\\
25	0.362264004099457\\
30	0.357194143256654\\
35	0.301780071536928\\
40	0.284994316185363\\
45	0.266970354850498\\
50	0.240610308027583\\
};

\addplot [color=blue,line width=1.5pt,mark=square]
  table[row sep=crcr]{%
5	0.553453322752088\\
10	0.471156638563995\\
15	0.405659464965516\\
20	0.372421082856211\\
25	0.335997452026574\\
30	0.31427860367719\\
35	0.28809584958384\\
40	0.281344411693283\\
45	0.27838105731384\\
50	0.257723595246451\\
};

\addplot [color=green,line width=1.5pt,mark=diamond]
  table[row sep=crcr]{%
5	0.553453322752088\\
10	0.561089949783087\\
15	0.458157381639904\\
20	0.424616594957831\\
25	0.408492583490474\\
30	0.391096985236472\\
35	0.335657217065409\\
40	0.409801611583529\\
45	0.293639075779904\\
50	0.295864175540893\\
};

\addplot [color=black,line width=1.5pt,mark=o]
  table[row sep=crcr]{%
5	0.553453322752088\\
10	0.469549598995476\\
15	0.425182004846978\\
20	0.379627579203797\\
25	0.352338250479379\\
30	0.319242952700123\\
35	0.286335433784398\\
40	0.269598350967355\\
45	0.264474366876646\\
50	0.247788497859541\\
};
\nextgroupplot[xlabel=rank-$k$, ylabel=time(s)]
\addplot [color=red,line width=1.5pt,mark=asterisk]
  table[row sep=crcr]{%
5	0.00083102\\
10	0.00232641\\
15	0.00719682\\
20	0.01185963\\
25	0.01873095\\
30	0.02729086\\
35	0.03490451\\
40	0.04343089\\
45	0.05269186\\
50	0.06544395\\
};
\addlegendentry{DEIM}

\addplot [color=blue,line width=1.5pt,mark=square]
  table[row sep=crcr]{%
5	0.00185439\\
10	0.00393858\\
15	0.00602717\\
20	0.01019511\\
25	0.0124101\\
30	0.01675018\\
35	0.01840286\\
40	0.02342047\\
45	0.02972271\\
50	0.0313212\\
};
\addlegendentry{B-DEIM-RRQR}

\addplot [color=green,line width=1.5pt,mark=diamond]
  table[row sep=crcr]{%
5	0.00177657\\
10	0.00347272\\
15	0.00466944\\
20	0.00836578\\
25	0.00943317\\
30	0.01338987\\
35	0.01883183\\
40	0.02598667\\
45	0.0267998\\
50	0.03136462\\
};
\addlegendentry{Q-DEIM}

\addplot [color=black,line width=1.5pt,mark=o]
  table[row sep=crcr]{%
5	0.00204999\\
10	0.00538739\\
15	0.00948512\\
20	0.01349167\\
25	0.0234334\\
30	0.03410716\\
35	0.03803654\\
40	0.04400741\\
45	0.05604693\\
50	0.05708153\\
};
\addlegendentry{AdapBlock-RRQR}

\end{groupplot}

\end{tikzpicture}}%
\caption{Relative approximation errors (left) and runtimes (right) as a function of $k$ for the block DEIM CUR approximation algorithms compared with some standard CUR approximation algorithms using the {\sf Jester} data set.\label{fig:3.2}}
\end{figure}

Based on the observations from \cref{fig:3.2}, we can conclude that the block DEIM methods generally provide slightly more accurate approximations compared to state-of-the-art methods. It is also important to note that the error of the {\sf MaxVol} and QDEIM approximations do not always decrease monotonically as the rank $k$ increases. When considering the runtime, both the B-DEIM-MaxVol and B-DEIM-RRQR algorithms demonstrate significantly lower computational times compared to the original DEIM scheme. In these small/mid-scale experiments, the adaptive variants of the block DEIM methods do not seem to improve runtimes compared to the DEIM procedure. There could be several reasons for this observation: the adaptive variants of block DEIM methods involve additional computations and operations compared to the standard DEIM procedure. These additional steps may introduce computational overhead that offsets the potential gains in runtime. In small/mid-scale scenarios, the overhead might outweigh the benefits. On the other hand, it is evident that the B-DEIM-MaxVol algorithm and its adaptive variant are more efficient than the standard {\sf MaxVol} approach. By utilizing the block DEIM variants, we gain improvements in both accuracy and speed compared to the standard {\sf MaxVol} method. Additionally, the B-DEIM-RRQR method proves to be equally efficient as the QDEIM procedure while providing a more accurate approximation here.
}
\end{experiment}

\begin{experiment}{\rm 
In the subsequent series of experiments, we turn our attention to evaluating the performance of our proposed block-DEIM variants when dealing with large-scale data. With a block size of $b=10$ selected for this particular set of experiments, our primary goal is to gain insights into how our block-DEIM approaches tackle the challenges presented by large-scale data sets and to assess their effectiveness and efficiency in this context. To conduct these evaluations, we utilize a set of standard test matrices specifically designed for sparse matrix problems. These data matrices are sourced from the publicly available SuiteSparse Matrix Collection. The diverse nature of these matrices allows us to assess the effectiveness of our approaches across various problem domains.

The first test matrix, referred to as {\sf g7jac100}, is derived from the ``Overlapping Generations Model'' used to study the social security systems of the G7 nations. It is a sparse matrix with dimensions $29610\times 29610$ and contains $335972$ numerically nonzero entries. Notably, this matrix has a low rank of $21971$. The second matrix, named {\sf net100}, originates from an optimization problem. It has dimensions of $29920\times 29920$ and contains $2033200$ numerically nonzero entries. Similar to the previous matrix, {\sf net100} also possesses a low rank, specifically $26983$. The {\sf Abacus-shell-ud} matrix, associated with model order reduction, has dimensions $23412\times 23412$ and represents a rank-$2048$ structure. It contains $218484$ nonzero entries. Lastly, we have {\sf pkustk01}, a symmetric positive-definite matrix derived from a civil engineering problem. This matrix has dimensions of $22044\times 22044$, a low rank of $3732$, and consists of $979380$ nonzero entries.

In the case of large-scale data sets, similar to the small/mid-scale experiments, in \cref{fig:3.21,fig:3.22,fig:3.23,fig:3.24} the block DEIM variants maintain comparable reconstruction errors as the existing methods. This finding aligns with our observations from the small/mid-scale experiments. However, there are notable differences in terms of algorithm efficiency.
Unlike the small/mid-scale cases, where the adaptive variants have similar runtimes as the DEIM scheme, in the large-scale experiments, the adaptive variants demonstrate better computational efficiency than the standard DEIM scheme. On the other hand, the B-DEIM-MaxVol and B-DEIM-RRQR schemes showcase better speed efficiency overall. These block DEIM variants prove to be effective in achieving a balance between accuracy and computational efficiency in the context of large-scale data. Consistent with previous findings, the MaxVol algorithm generally remains the least efficient method, with one exception in the case of the {\sf g7jac100} data set.
These results highlight the importance of considering the specific characteristics and requirements of the data sets when selecting an appropriate algorithm.

\begin{figure}[htb!]
    \centering
%
%
\definecolor{mycolor1}{rgb}{0.00000,0.44700,0.74100}%
\definecolor{mycolor2}{rgb}{0.85000,0.32500,0.09800}%
\pgfplotsset{every axis title/.append style={at={(1,1)}}}
{\begin{tikzpicture}

\begin{groupplot}[
group style={
    group name=my plots,
    group size=2 by 2,
    xlabels at=edge bottom,
    ylabels at=edge left,
    horizontal sep=1.5cm,
    vertical sep=3cm,
    },
     xmin=100,
     xmax=500,
    legend style={at={(-0.2,-0.5)},anchor=south,legend  columns =4},
width=0.5\linewidth,
height=5cm
]

\nextgroupplot[title=MaxVol methods,xlabel=rank-$k$, ylabel= {$\| A - CMR \|\ /\ \| A \|$}]
\addplot [color=red,line width=1.5pt,mark=asterisk]
  table[row sep=crcr]{%
100	0.518471507862923\\
120	0.493971812175609\\
140	0.461240032105582\\
160	0.432597975789714\\
180	0.377563176141657\\
200	0.369512550201253\\
220	0.312698455220967\\
240	0.297557231089682\\
260	0.285650094731187\\
280	0.27102440420121\\
300	0.251480219806051\\
320	0.249434374574061\\
340	0.229105882454455\\
360	0.221666734385136\\
380	0.216947295778839\\
400	0.209715219208256\\
420	0.198039899215204\\
440	0.196302595152347\\
460	0.19329559541864\\
480	0.188624527154271\\
500	0.184295592438967\\
};

\addplot [color=blue,line width=1.5pt,mark=square]
  table[row sep=crcr]{%
100	0.558861646380679\\
120	0.474319118658313\\
140	0.451183788546649\\
160	0.386718609010585\\
180	0.364090262869298\\
200	0.337365070697103\\
220	0.328619427993173\\
240	0.309842077083575\\
260	0.297125490910785\\
280	0.275150569570335\\
300	0.26928896635232\\
320	0.235922553424094\\
340	0.230543591740065\\
360	0.223312720975655\\
380	0.210735947377604\\
400	0.208068263121497\\
420	0.198041080098301\\
440	0.196304490853212\\
460	0.193295587548082\\
480	0.188624926842113\\
500	0.184295980482613\\
};

\addplot [color=green,line width=1.5pt,mark=diamond]
  table[row sep=crcr]{%
100	0.541811236388936\\
120	0.495209114740307\\
140	0.448524468778537\\
160	0.415218257492386\\
180	0.433764020234211\\
200	0.451919120462753\\
220	0.472634809530381\\
240	0.485189397469178\\
260	0.495423486320642\\
280	0.332554655278211\\
300	0.508499732344035\\
320	0.358008706289369\\
340	0.39812108790721\\
360	0.376855410928887\\
380	0.378487848554036\\
400	0.307062648095574\\
420	0.309944454151539\\
440	0.311210483592218\\
460	0.263858535954967\\
480	0.258164723181026\\
500	0.260784596094849\\
};

\addplot [color=black,line width=1.5pt,mark=o]
  table[row sep=crcr]{%
100	0.561129445218334\\
120	0.520106235595328\\
140	0.452754373741477\\
160	0.386718609010585\\
180	0.366406581162439\\
200	0.337365899891732\\
220	0.327092904690799\\
240	0.311239618125877\\
260	0.29712531088274\\
280	0.275148786424299\\
300	0.26928896635232\\
320	0.236552891601408\\
340	0.23054389878244\\
360	0.223312720975655\\
380	0.210735947377604\\
400	0.206640015662598\\
420	0.198041050884464\\
440	0.196304413001838\\
460	0.193295587548082\\
480	0.188624926842113\\
500	0.184295980482613\\
};
\nextgroupplot[xlabel=rank-$k$,ylabel=time(s)]

\addplot [color=red,line width=1.5pt,mark=asterisk]
  table[row sep=crcr]{%
100	0.738109\\
120	0.9671176\\
140	1.3510901\\
160	1.6416619\\
180	2.0596095\\
200	2.9031888\\
220	3.3209467\\
240	4.082354\\
260	4.9187129\\
280	5.7175135\\
300	6.5463422\\
320	7.4816507\\
340	8.2590442\\
360	9.3901307\\
380	10.3276131\\
400	11.4871543\\
420	12.6961873\\
440	13.7997056\\
460	14.5127016\\
480	16.2883833\\
500	18.2936792\\
};
\addlegendentry{DEIM}

\addplot [color=blue,line width=1.5pt,mark=square]
  table[row sep=crcr]{%
100	0.2115235\\
120	0.2171567\\
140	0.2565045\\
160	0.288343\\
180	0.3409674\\
200	0.3942068\\
220	0.4637455\\
240	0.5264475\\
260	0.5750367\\
280	0.6597199\\
300	0.7275896\\
320	0.794588\\
340	0.8838613\\
360	0.9815979\\
380	1.1189922\\
400	1.1814202\\
420	1.2196211\\
440	1.2717791\\
460	1.359176\\
480	1.4715439\\
500	1.5838737\\
};
\addlegendentry{B-DEIM-MaxVol}

\addplot [color=green,line width=1.5pt,mark=diamond]
  table[row sep=crcr]{%
100	0.3206484\\
120	0.4816991\\
140	0.3673486\\
160	0.506684\\
180	0.5752395\\
200	0.7350836\\
220	0.9576913\\
240	0.9105379\\
260	0.9521985\\
280	1.3222234\\
300	1.3814935\\
320	1.6490489\\
340	1.9471079\\
360	2.1313716\\
380	2.517576\\
400	2.8122197\\
420	2.4111286\\
440	2.5268093\\
460	3.0337009\\
480	3.1830234\\
500	4.4285225\\
};
\addlegendentry{MaxVol}

\addplot [color=black,line width=1.5pt,mark=o]
  table[row sep=crcr]{%
100	0.3981347\\
120	0.6056128\\
140	0.8690831\\
160	0.8740424\\
180	1.1657947\\
200	1.329588\\
220	1.6430587\\
240	2.1713765\\
260	2.3866258\\
280	2.8543146\\
300	2.9893467\\
320	3.1505885\\
340	3.5214008\\
360	4.3314004\\
380	4.4245298\\
400	5.3467325\\
420	5.7518832\\
440	6.1928015\\
460	7.0446944\\
480	7.6540806\\
500	7.7963173\\
};
\addlegendentry{AdapBlock-MaxVol}

\nextgroupplot[title=RRQR methods, xlabel=rank-$k$,ylabel= {$\| A - CMR \|\ /\ \| A \|$}]

\addplot [color=red,line width=1.5pt,mark=asterisk]
  table[row sep=crcr]{%
100	0.518471507862923\\
120	0.493971812175609\\
140	0.461240032105582\\
160	0.432597975789714\\
180	0.377563176141657\\
200	0.369512550201253\\
220	0.312698455220967\\
240	0.297557231089682\\
260	0.285650094731187\\
280	0.27102440420121\\
300	0.251480219806051\\
320	0.249434374574061\\
340	0.229105882454455\\
360	0.221666734385136\\
380	0.216947295778839\\
400	0.209715219208256\\
420	0.198039899215204\\
440	0.196302595152347\\
460	0.19329559541864\\
480	0.188624527154271\\
500	0.184295592438967\\
};

\addplot [color=blue,line width=1.5pt,mark=square]
  table[row sep=crcr]{%
100	0.561579209387949\\
120	0.520030903157774\\
140	0.452234452588656\\
160	0.383134922459981\\
180	0.364405890276683\\
200	0.349706769331529\\
220	0.3257564133634\\
240	0.313825712429049\\
260	0.30176899060446\\
280	0.279752068629662\\
300	0.273307929140647\\
320	0.233609477124332\\
340	0.227668790161977\\
360	0.220677709067463\\
380	0.209199336831561\\
400	0.207823778444469\\
420	0.198040682046974\\
440	0.19630393237425\\
460	0.193318833483545\\
480	0.188659599497146\\
500	0.184308995070822\\
};

\addplot [color=green,line width=1.5pt,mark=diamond]
  table[row sep=crcr]{%
100	0.549359961674411\\
120	0.498082865750254\\
140	0.450065194102351\\
160	0.413175939879289\\
180	0.432827793622456\\
200	0.451462467689926\\
220	0.47273686675168\\
240	0.485455716783496\\
260	0.364013814906252\\
280	0.328755734574091\\
300	0.508501801217514\\
320	0.36155826901847\\
340	0.36766176006494\\
360	0.374815250810793\\
380	0.378520956510429\\
400	0.290841520527453\\
420	0.309986384247543\\
440	0.311209876740541\\
460	0.246624275249654\\
480	0.25019140358422\\
500	0.253102550586343\\
};

\addplot [color=black,line width=1.5pt,mark=o]
  table[row sep=crcr]{%
100	0.561572909735467\\
120	0.519994413027964\\
140	0.450800037934766\\
160	0.382055096795171\\
180	0.362093579845303\\
200	0.347109411716835\\
220	0.326844389564443\\
240	0.31308695686336\\
260	0.30032968930251\\
280	0.276796093504028\\
300	0.261917742861471\\
320	0.255246444109639\\
340	0.224023667212074\\
360	0.219668422362017\\
380	0.215229017824455\\
400	0.206075398364005\\
420	0.203056160773443\\
440	0.19630412413943\\
460	0.193295443105354\\
480	0.188659752347161\\
500	0.184348075519889\\
};
\nextgroupplot[xlabel=rank-$k$, ylabel=time(s)]
\addplot [color=red,line width=1.5pt,mark=asterisk]
  table[row sep=crcr]{%
100	0.738109\\
120	0.9671176\\
140	1.3510901\\
160	1.6416619\\
180	2.0596095\\
200	2.9031888\\
220	3.3209467\\
240	4.082354\\
260	4.9187129\\
280	5.7175135\\
300	6.5463422\\
320	7.4816507\\
340	8.2590442\\
360	9.3901307\\
380	10.3276131\\
400	11.4871543\\
420	12.6961873\\
440	13.7997056\\
460	14.5127016\\
480	16.2883833\\
500	18.2936792\\
};
\addlegendentry{DEIM}

\addplot [color=blue,line width=1.5pt,mark=square]
  table[row sep=crcr]{%
100	0.1963719\\
120	0.2185565\\
140	0.2823946\\
160	0.3164598\\
180	0.3574829\\
200	0.4023751\\
220	0.4483645\\
240	0.5085465\\
260	0.5807693\\
280	0.6248059\\
300	0.703557\\
320	0.7620836\\
340	0.8331746\\
360	0.9117284\\
380	1.0065283\\
400	1.0703847\\
420	1.1644077\\
440	1.2505647\\
460	1.3594201\\
480	1.4323872\\
500	1.5579722\\
};
\addlegendentry{B-DEIM-RRQR}

\addplot [color=green,line width=1.5pt,mark=diamond]
  table[row sep=crcr]{%
100	0.1655535\\
120	0.2145025\\
140	0.2677999\\
160	0.3554147\\
180	0.5169549\\
200	0.6143796\\
220	0.7390828\\
240	0.9424734\\
260	1.0694075\\
280	1.2741538\\
300	1.4557652\\
320	1.7745532\\
340	1.871172\\
360	2.1365271\\
380	2.3573583\\
400	2.6752952\\
420	2.8798791\\
440	3.0900393\\
460	3.372058\\
480	3.7581564\\
500	4.0157662\\
};
\addlegendentry{Q-DEIM}

\addplot [color=black,line width=1.5pt,mark=o]
  table[row sep=crcr]{%
100	0.4052859\\
120	0.6055837\\
140	0.7463877\\
160	0.7688278\\
180	1.0610489\\
200	1.229024\\
220	1.5835553\\
240	2.0807258\\
260	2.4960747\\
280	2.9078948\\
300	3.0847168\\
320	3.2866949\\
340	4.0181666\\
360	4.8593673\\
380	4.9698212\\
400	5.4438014\\
420	6.2206309\\
440	6.9906123\\
460	7.5055744\\
480	8.1933515\\
500	8.3881958\\
};
\addlegendentry{AdapBlock-RRQR}

\end{groupplot}

\end{tikzpicture}}%
    \caption{Relative approximation errors (left) and runtimes (right) as a function of $k$ for the block DEIM CUR approximation algorithms compared with some standard CUR approximation algorithms using the {\sf g7jac100} sparse matrix.\label{fig:3.21}}  
\end{figure}
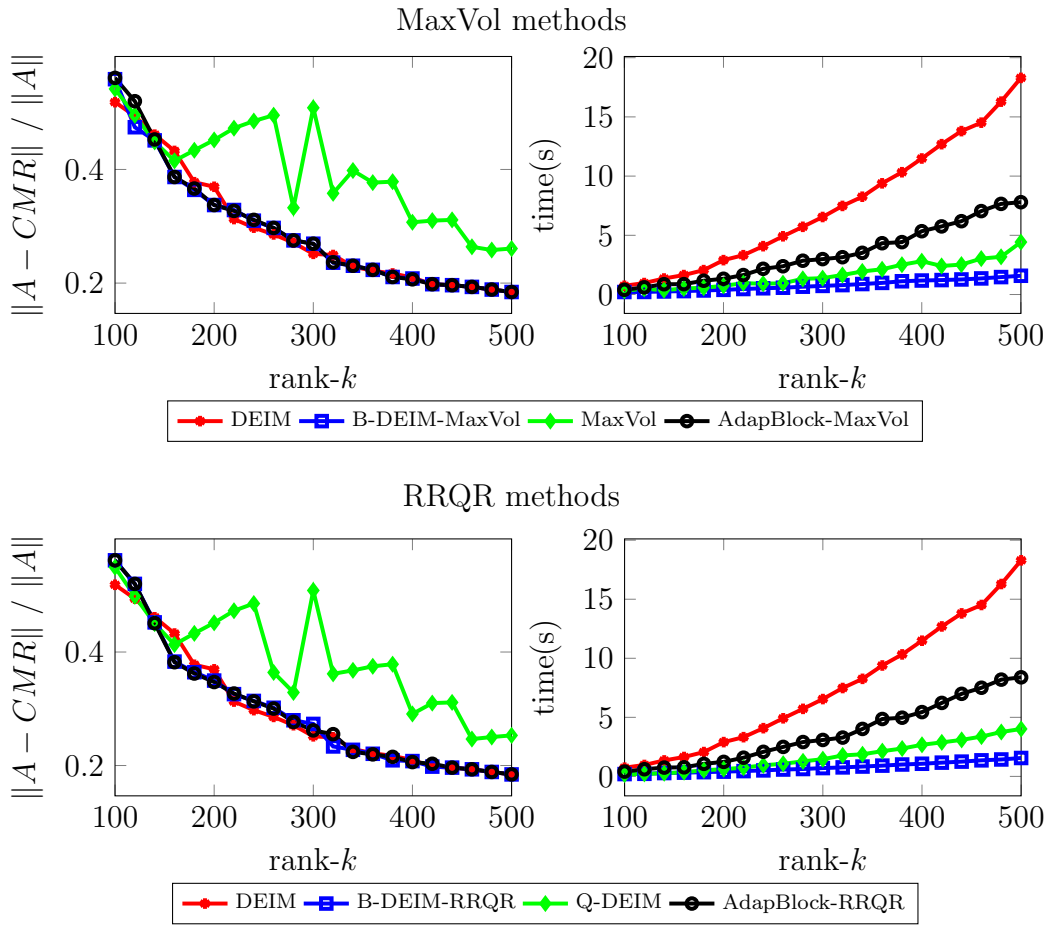

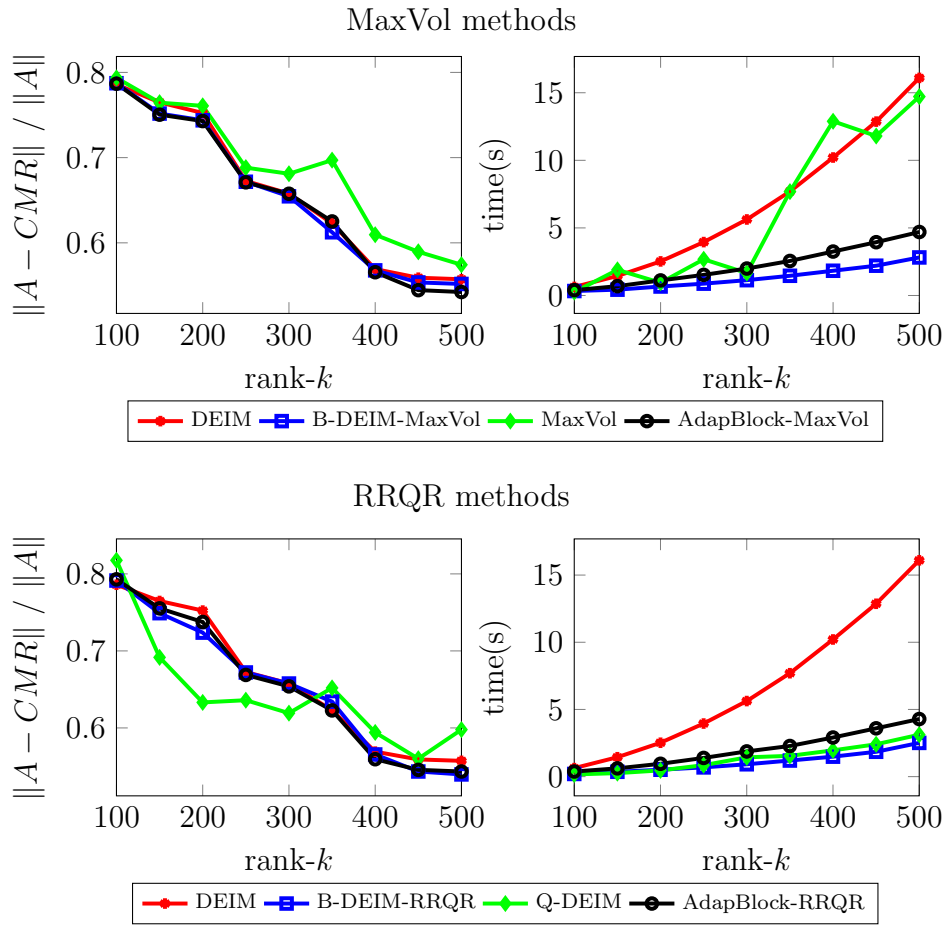
\begin{figure}[htb!]
    \centering
%
%
\definecolor{mycolor1}{rgb}{0.00000,0.44700,0.74100}%
\definecolor{mycolor2}{rgb}{0.85000,0.32500,0.09800}%
\pgfplotsset{every axis title/.append style={at={(1,1)}}}
{\begin{tikzpicture}

\begin{groupplot}[
group style={
    group name=my plots,
    group size=2 by 2,
    xlabels at=edge bottom,
    ylabels at=edge left,
    horizontal sep=1.5cm,
    vertical sep=3cm,
    },
     xmin=100,
     xmax=500,
    legend style={at={(-0.2,-0.5)},anchor=south,legend  columns =4},
width=0.45\linewidth,
height=5cm
]

\nextgroupplot[title=MaxVol methods,xlabel=rank-$k$, ylabel= {$\| A - CMR \|\ /\ \| A \|$}]
\addplot [color=red,line width=1.5pt,mark=asterisk]
  table[row sep=crcr]{%
100	0.785915748941538\\
150	0.764812912007411\\
200	0.752361050303119\\
250	0.672632506676465\\
300	0.657784643908181\\
350	0.622867389715398\\
400	0.569211887583689\\
450	0.559070533335692\\
500	0.557277235614791\\
};

\addplot [color=blue,line width=1.5pt,mark=square]
  table[row sep=crcr]{%
100	0.787319305776992\\
150	0.751705508803443\\
200	0.743790028781915\\
250	0.671687479234122\\
300	0.654610523738024\\
350	0.612411334280278\\
400	0.567704044787977\\
450	0.553516997009367\\
500	0.551764937136436\\
};

\addplot [color=green,line width=1.5pt,mark=diamond]
  table[row sep=crcr]{%
100	0.793549401303383\\
150	0.76449486316263\\
200	0.760741261842174\\
250	0.688249721921323\\
300	0.681059305919705\\
350	0.697044365049344\\
400	0.609525527066757\\
450	0.58950292422982\\
500	0.574299514675946\\
};

\addplot [color=black,line width=1.5pt,mark=o]
  table[row sep=crcr]{%
100	0.786746395283589\\
150	0.750406099262703\\
200	0.742889420452976\\
250	0.671029070096899\\
300	0.657589314014104\\
350	0.625107353643479\\
400	0.565517013400189\\
450	0.544516070715524\\
500	0.542300031586324\\
};
\nextgroupplot[xlabel=rank-$k$,ylabel=time(s)]

\addplot [color=red,line width=1.5pt,mark=asterisk]
  table[row sep=crcr]{%
100	0.6326047\\
150	1.4441278\\
200	2.5214358\\
250	3.9510294\\
300	5.6242923\\
350	7.6960595\\
400	10.2072181\\
450	12.8639043\\
500	16.0980154\\
};
\addlegendentry{DEIM}

\addplot [color=blue,line width=1.5pt,mark=square]
  table[row sep=crcr]{%
100	0.3306703\\
150	0.4309458\\
200	0.6640768\\
250	0.8760142\\
300	1.1339923\\
350	1.4586857\\
400	1.8309039\\
450	2.2047306\\
500	2.8149707\\
};
\addlegendentry{B-DEIM-MaxVol}

\addplot [color=green,line width=1.5pt,mark=diamond]
  table[row sep=crcr]{%
100	0.2579957\\
150	1.8961476\\
200	0.9454121\\
250	2.6792389\\
300	1.6593337\\
350	7.6853569\\
400	12.8888449\\
450	11.7984686\\
500	14.7135653\\
};
\addlegendentry{MaxVol}

\addplot [color=black,line width=1.5pt,mark=o]
  table[row sep=crcr]{%
100	0.3949428\\
150	0.6907796\\
200	1.1152832\\
250	1.5249473\\
300	1.994186\\
350	2.5553943\\
400	3.2508483\\
450	3.9423278\\
500	4.6969872\\
};
\addlegendentry{AdapBlock-MaxVol}

\nextgroupplot[title=RRQR methods, xlabel=rank-$k$,ylabel= {$\| A - CMR \|\ /\ \| A \|$}]

\addplot [color=red,line width=1.5pt,mark=asterisk]
  table[row sep=crcr]{%
100	0.785915748941538\\
150	0.764812912007411\\
200	0.752361050303119\\
250	0.672632506676465\\
300	0.657784643908181\\
350	0.622867389715398\\
400	0.569211887583689\\
450	0.559070533335692\\
500	0.557277235614791\\
};

\addplot [color=blue,line width=1.5pt,mark=square]
  table[row sep=crcr]{%
100	0.791364094903062\\
150	0.748998973874715\\
200	0.723753663165394\\
250	0.671999751374696\\
300	0.657650923189572\\
350	0.634171121496038\\
400	0.565549622360929\\
450	0.543431259272207\\
500	0.539622045834168\\
};

\addplot [color=green,line width=1.5pt,mark=diamond]
  table[row sep=crcr]{%
100	0.817612222796921\\
150	0.691583543583441\\
200	0.633132181290521\\
250	0.635994871261912\\
300	0.619209210167121\\
350	0.651937166430743\\
400	0.594144702287352\\
450	0.560024943882168\\
500	0.597834542471622\\
};

\addplot [color=black,line width=1.5pt,mark=o]
  table[row sep=crcr]{%
100	0.792940531007331\\
150	0.75525042003326\\
200	0.737546742209694\\
250	0.668812191176654\\
300	0.653950616625597\\
350	0.622739265945017\\
400	0.559380891679103\\
450	0.545600505090727\\
500	0.54343312166475\\
};
\nextgroupplot[xlabel=rank-$k$, ylabel=time(s)]
\addplot [color=red,line width=1.5pt,mark=asterisk]
  table[row sep=crcr]{%
100	0.6326047\\
150	1.4441278\\
200	2.5214358\\
250	3.9510294\\
300	5.6242923\\
350	7.6960595\\
400	10.2072181\\
450	12.8639043\\
500	16.0980154\\
};
\addlegendentry{DEIM}

\addplot [color=blue,line width=1.5pt,mark=square]
  table[row sep=crcr]{%
100	0.1993483\\
150	0.3602156\\
200	0.513326\\
250	0.6879016\\
300	0.9194383\\
350	1.1977879\\
400	1.4859106\\
450	1.8491024\\
500	2.5110869\\
};
\addlegendentry{B-DEIM-RRQR}

\addplot [color=green,line width=1.5pt,mark=diamond]
  table[row sep=crcr]{%
100	0.1638751\\
150	0.2607939\\
200	0.4640347\\
250	0.8530066\\
300	1.4352998\\
350	1.5465362\\
400	1.9550306\\
450	2.4075371\\
500	3.119051\\
};
\addlegendentry{Q-DEIM}

\addplot [color=black,line width=1.5pt,mark=o]
  table[row sep=crcr]{%
100	0.3801715\\
150	0.6158495\\
200	0.9556588\\
250	1.3916556\\
300	1.8771699\\
350	2.2797177\\
400	2.9111325\\
450	3.5946579\\
500	4.2853553\\
};
\addlegendentry{AdapBlock-RRQR}

\end{groupplot}

\end{tikzpicture}}%
    \caption{Relative approximation errors (left) and runtimes (right) as a function of $k$ for the block DEIM CUR approximation algorithms compared with some standard CUR approximation algorithms using the {\sf net100} sparse matrix.\label{fig:3.22}}    
\end{figure}
\begin{figure}[htb!]
    \centering
%
%
\definecolor{mycolor1}{rgb}{0.00000,0.44700,0.74100}%
\definecolor{mycolor2}{rgb}{0.85000,0.32500,0.09800}%
\pgfplotsset{every axis title/.append style={at={(1,1)}}}
{\begin{tikzpicture}

\begin{groupplot}[
group style={
    group name=my plots,
    group size=2 by 2,
    xlabels at=edge bottom,
    ylabels at=edge left,
    horizontal sep=1.5cm,
    vertical sep=3cm,
    },
     xmin=100,
     xmax=500,
   legend style={at={(-0.2,-0.5)},anchor=south,legend  columns =4},
width=0.45\linewidth,
height=5cm
]

\nextgroupplot[title=MaxVol methods,xlabel=rank-$k$, ylabel= {$\| A - CMR \|\ /\ \| A \|$}]
\addplot [color=red, mark=asterisk, mark options={solid, red}, line width=1.5pt]
  table[row sep=crcr]{%
100	0.731006176062257\\
150	0.675429870137618\\
200	0.603736934375861\\
250	0.495962568088575\\
300	0.46323483429763\\
350	0.448651872976158\\
400	0.4367203319692\\
450	0.421607185245341\\
500	0.393615935529745\\
};
\addplot [color=blue, mark=square, mark options={solid, blue}, line width=1.5pt]
  table[row sep=crcr]{%
100	0.731312530380574\\
150	0.639628266124923\\
200	0.606449739349624\\
250	0.514166377759352\\
300	0.469459980637182\\
350	0.456517389568782\\
400	0.456259158302393\\
450	0.405236202298007\\
500	0.384750615714942\\
};
\addplot [color=green, mark=diamond, mark options={solid, green}, line width=1.5pt]
  table[row sep=crcr]{%
100	0.673514647368648\\
150	0.551902219963793\\
200	0.521944979620028\\
250	0.495868928487359\\
300	0.479036900248891\\
350	0.44849704035231\\
400	0.436429529798299\\
450	0.394783514083698\\
500	0.367932541543967\\
};
\addplot [color=black, mark=o, mark options={solid, black}, line width=1.5pt]
  table[row sep=crcr]{%
100	0.731312530380574\\
150	0.639628266124919\\
200	0.606444588017904\\
250	0.497697740796272\\
300	0.471149063148321\\
350	0.446827648729365\\
400	0.440189618603804\\
450	0.404856335310227\\
500	0.38129476838518\\
};
\nextgroupplot[xlabel=rank-$k$,ylabel=time(s)]

\addplot [color=red, mark=asterisk, mark options={solid, red},line width=1.5pt]
  table[row sep=crcr]{%
100	0.6304604\\
150	1.3742761\\
200	2.4410774\\
250	3.7826641\\
300	5.4943013\\
350	7.45935\\
400	9.7615815\\
450	12.4115156\\
500	15.4153137\\
};
\addlegendentry{DEIM}

\addplot [color=blue, mark=square, mark options={solid, blue},line width=1.5pt]
  table[row sep=crcr]{%
100	0.3085502\\
150	0.4844009\\
200	0.6866624\\
250	0.9043605\\
300	1.1704532\\
350	1.4456309\\
400	1.7834823\\
450	2.1260178\\
500	2.563097\\
};
\addlegendentry{B-DEIM-MaxVol}

\addplot [color=green, mark=diamond, mark options={solid, green},line width=1.5pt]
  table[row sep=crcr]{%
100	2.0335299\\
150	3.9611879\\
200	3.5136555\\
250	4.826744\\
300	8.7266944\\
350	11.9125421\\
400	15.7511643\\
450	17.9496706\\
500	23.225687\\
};
\addlegendentry{MaxVol}

\addplot [color=black, mark=o, mark options={solid, black},line width=1.5pt]
  table[row sep=crcr]{%
100	0.3791572\\
150	0.7780111\\
200	1.1188292\\
250	1.5173956\\
300	2.2158121\\
350	2.7603108\\
400	3.3662135\\
450	4.4621321\\
500	5.1922435\\
};
\addlegendentry{AdapBlock-MaxVol}

\nextgroupplot[title=RRQR methods, xlabel=rank-$k$,ylabel= {$\| A - CMR \|\ /\ \| A \|$}]

\addplot [color=red, mark=asterisk, mark options={solid, red}, line width =1.5pt]
  table[row sep=crcr]{%
100	0.731006176062257\\
150	0.675429870137618\\
200	0.603736934375861\\
250	0.495962568088575\\
300	0.46323483429763\\
350	0.448651872976158\\
400	0.4367203319692\\
450	0.421607185245341\\
500	0.393615935529745\\
};
\addplot [color=blue, mark=square, mark options={solid, blue}, line width =1.5pt]
  table[row sep=crcr]{%
100	0.717490761794574\\
150	0.680083581146623\\
200	0.604761072231738\\
250	0.51329705280673\\
300	0.466273761665729\\
350	0.442489379423096\\
400	0.438198825001182\\
450	0.400721226137398\\
500	0.384034261979663\\
};
\addplot [color=green, mark=diamond, mark options={solid, green}, line width =1.5pt]
  table[row sep=crcr]{%
100	0.649195314820347\\
150	0.540298351970616\\
200	0.555787376583647\\
250	0.486653636957656\\
300	0.463065712513251\\
350	0.438012816299073\\
400	0.415046884544927\\
450	0.386791747240483\\
500	0.415520105112322\\
};
\addplot [color=black, mark=o, mark options={solid, black}, line width =1.5pt]
  table[row sep=crcr]{%
100	0.717490761794574\\
150	0.680083945072015\\
200	0.60640817841195\\
250	0.505726784010023\\
300	0.466353653376281\\
350	0.454372843598452\\
400	0.443901370338997\\
450	0.405841071931926\\
500	0.369685310481709\\
};
\nextgroupplot[xlabel=rank-$k$, ylabel=time(s)]
\addplot [color=red, mark=asterisk, mark options={solid, red},line width =1.5pt]
  table[row sep=crcr]{%
100	0.6304604\\
150	1.3742761\\
200	2.4410774\\
250	3.7826641\\
300	5.4943013\\
350	7.45935\\
400	9.7615815\\
450	12.4115156\\
500	15.4153137\\
};
\addlegendentry{DEIM}

\addplot [color=blue, mark=square, mark options={solid, blue},line width =1.5pt]
  table[row sep=crcr]{%
100	0.1817532\\
150	0.3127434\\
200	0.4763983\\
250	0.6795662\\
300	0.8903106\\
350	1.1642954\\
400	1.4742678\\
450	1.8205858\\
500	2.1884504\\
};
\addlegendentry{B-DEIM-RRQR}

\addplot [color=green, mark=diamond, mark options={solid, green},line width =1.5pt]
  table[row sep=crcr]{%
100	0.1227004\\
150	0.1926106\\
200	0.2996184\\
250	0.4875772\\
300	0.7003834\\
350	1.0140221\\
400	1.3608423\\
450	1.7669146\\
500	2.1481623\\
};
\addlegendentry{Q-DEIM}

\addplot [color=black, mark=o, mark options={solid, black},line width =1.5pt]
  table[row sep=crcr]{%
100	0.2531509\\
150	0.5990439\\
200	0.8979281\\
250	1.2836183\\
300	1.7324341\\
350	2.1918368\\
400	3.2005734\\
450	3.7633641\\
500	4.9095401\\
};
\addlegendentry{AdapBlock-RRQR}

\end{groupplot}

\end{tikzpicture}}%
    \caption{Relative approximation errors (left) and runtimes (right) as a function of $k$ for the block DEIM CUR approximation algorithms compared with some standard CUR approximation algorithms using the {\sf Abacusa-shell-ud} sparse matrix.\label{fig:3.23}}    
\end{figure}

\begin{figure}[htb!]
    \centering
%
%
\definecolor{mycolor1}{rgb}{0.00000,0.44700,0.74100}%
\definecolor{mycolor2}{rgb}{0.85000,0.32500,0.09800}%
\pgfplotsset{every axis title/.append style={at={(1,1)}}}
{\begin{tikzpicture}

\begin{groupplot}[
group style={
    group name=my plots,
    group size=2 by 2,
    xlabels at=edge bottom,
    ylabels at=edge left,
    horizontal sep=1.5cm,
    vertical sep=3cm,
    },
     xmin=100,
     xmax=500,
    legend style={at={(-0.2,-0.5)},anchor=south,legend  columns =4},
width=0.45\linewidth,
height=5cm
]

\nextgroupplot[title=MaxVol methods,xlabel=rank-$k$, ylabel= {$\| A - CMR \|\ /\ \| A \|$}]
\addplot [color=red,line width=1.5pt,mark=asterisk]
  table[row sep=crcr]{%
100	0.835667090832298\\
150	0.754829812331701\\
200	0.704786979352574\\
250	0.693968916757072\\
300	0.663889489935972\\
350	0.634375002205603\\
400	0.625309970820682\\
450	0.625339858267061\\
500	0.561702116881604\\
};

\addplot [color=blue,line width=1.5pt,mark=square]
  table[row sep=crcr]{%
100	0.80890646981547\\
150	0.740475587621965\\
200	0.707819999536479\\
250	0.699206987631314\\
300	0.637924069647191\\
350	0.618302579334822\\
400	0.592584217597319\\
450	0.588134905683175\\
500	0.571405625458529\\
};

\addplot [color=green,line width=1.5pt,mark=diamond]
  table[row sep=crcr]{%
100	0.818438833489536\\
150	0.722406478674924\\
200	0.702172056703871\\
250	0.702918372593783\\
300	0.652378485033405\\
350	0.584864195273961\\
400	0.572620027571593\\
450	0.573074624381159\\
500	0.565599017199298\\
};

\addplot [color=black,line width=1.5pt,mark=o]
  table[row sep=crcr]{%
100	0.80890646981547\\
150	0.740475587621965\\
200	0.707819999536479\\
250	0.699206987631314\\
300	0.637924069647191\\
350	0.618302579334822\\
400	0.592584217597319\\
450	0.588134905683175\\
500	0.571405625458529\\
};
\nextgroupplot[xlabel=rank-$k$,ylabel=time(s)]

\addplot [color=red,line width=1.5pt,mark=asterisk]
  table[row sep=crcr]{%
100	0.6813474\\
150	1.555798\\
200	2.7087483\\
250	4.0408611\\
300	5.8659706\\
350	8.1640449\\
400	11.3032215\\
450	14.3891896\\
500	17.656137\\
};
\addlegendentry{DEIM}

\addplot [color=blue,line width=1.5pt,mark=square]
  table[row sep=crcr]{%
100	0.2925055\\
150	0.5153048\\
200	0.7427635\\
250	0.9962476\\
300	1.2383762\\
350	1.5970881\\
400	1.9796111\\
450	2.3519816\\
500	2.8806525\\
};
\addlegendentry{B-DEIM-MaxVol}

\addplot [color=green,line width=1.5pt,mark=diamond]
  table[row sep=crcr]{%
100	1.9768989\\
150	3.4647431\\
200	4.526783\\
250	6.6137329\\
300	10.9016395\\
350	12.9612844\\
400	16.4766828\\
450	21.6026201\\
500	26.4780213\\
};
\addlegendentry{MaxVol}

\addplot [color=black,line width=1.5pt,mark=o]
  table[row sep=crcr]{%
100	0.4115532\\
150	0.6451238\\
200	1.0085611\\
250	1.3800459\\
300	1.8893044\\
350	2.4264338\\
400	3.0201004\\
450	3.7088324\\
500	4.5069163\\
};
\addlegendentry{AdapBlock-MaxVol}

\nextgroupplot[title=RRQR methods, xlabel=rank-$k$,ylabel= {$\| A - CMR \|\ /\ \| A \|$}]

\addplot [color=red,line width=1.5pt,mark=asterisk]
  table[row sep=crcr]{%
100	0.835667090832298\\
150	0.754829812331701\\
200	0.704786979352574\\
250	0.693968916757072\\
300	0.663889489935972\\
350	0.634375002205603\\
400	0.625309970820682\\
450	0.625339858267061\\
500	0.561702116881604\\
};

\addplot [color=blue,line width=1.5pt,mark=square]
  table[row sep=crcr]{%
100	0.815761567043561\\
150	0.734838855942786\\
200	0.704264100567504\\
250	0.671402715648435\\
300	0.64766300705507\\
350	0.616517986563206\\
400	0.587700742658633\\
450	0.575176062074754\\
500	0.553440491192616\\
};

\addplot [color=green,line width=1.5pt,mark=diamond]
  table[row sep=crcr]{%
100	0.820978470772503\\
150	0.72520081398979\\
200	0.710879833894972\\
250	0.682269863707333\\
300	0.639394129625334\\
350	0.619586664569729\\
400	0.575759483605372\\
450	0.574958830757008\\
500	0.569927493379126\\
};

\addplot [color=black,line width=1.5pt,mark=o]
  table[row sep=crcr]{%
100	0.815761567043561\\
150	0.734838855942786\\
200	0.704264100567504\\
250	0.671402715648435\\
300	0.64766300705507\\
350	0.616517986563206\\
400	0.587700742658633\\
450	0.575176062074754\\
500	0.553440491192616\\
};
\nextgroupplot[xlabel=rank-$k$, ylabel=time(s)]
\addplot [color=red,line width=1.5pt,mark=asterisk]
  table[row sep=crcr]{%
100	0.6813474\\
150	1.555798\\
200	2.7087483\\
250	4.0408611\\
300	5.8659706\\
350	8.1640449\\
400	11.3032215\\
450	14.3891896\\
500	17.656137\\
};
\addlegendentry{DEIM}

\addplot [color=blue,line width=1.5pt,mark=square]
  table[row sep=crcr]{%
100	0.1968704\\
150	0.3236936\\
200	0.5516448\\
250	0.7043459\\
300	1.000145\\
350	1.2918529\\
400	1.651528\\
450	2.0038767\\
500	2.4728857\\
};
\addlegendentry{B-DEIM-RRQR}

\addplot [color=green,line width=1.5pt,mark=diamond]
  table[row sep=crcr]{%
100	0.3116483\\
150	0.6621166\\
200	1.1374546\\
250	1.739453\\
300	2.406961\\
350	3.1781154\\
400	4.1184946\\
450	5.1760309\\
500	6.2496314\\
};
\addlegendentry{Q-DEIM}

\addplot [color=black,line width=1.5pt,mark=o]
  table[row sep=crcr]{%
100	0.2686286\\
150	0.5352267\\
200	0.8294175\\
250	1.1729687\\
300	1.5739364\\
350	2.0976771\\
400	2.6937188\\
450	3.3358207\\
500	4.1027333\\
};
\addlegendentry{AdapBlock-RRQR}

\end{groupplot}

\end{tikzpicture}}%
    \caption{Relative approximation errors (left) and runtimes (right) as a function of $k$ for the block DEIM CUR approximation algorithms compared with some standard CUR approximation algorithms using the {\sf pkustk01} sparse matrix.\label{fig:3.24}}   
\end{figure}

\begin{table}[htb!]
\footnotesize \centering
\caption{Comparison of the error and runtimes for selecting $k=500$ columns and rows (the runtimes include time for computing the rank-$k$ SVD) for the various algorithms using the large scale data sets.}\label{tab:lascale}
\begin{tabular}{l|ccc|ccc|cc}\hline \rule{0pt}{2.3ex}%
Method $\backslash$ Data           & \multicolumn{2}{c}{\sf net100}        & & \multicolumn{2}{c}{\sf pkustk01}   & & \multicolumn{2}{c}{\sf Abacus-shell-ud}\\ 
                  & Error & Time (s)      &   & Error & Time (s)     &   & Error & Time (s)     \\ \hline \rule{0pt}{2.3ex}%
SVD     &0.38&79&&0.07&69&&0.11&54\\[0.5mm] \hdashline
&&&&&& \\[-3mm] 
DEIM     &0.56 &99&&0.18&88&&0.39&69\\
QDEIM     &0.59 &84&&0.25&73&&0.42&56\\
{\sf MaxVol}&0.57 &97&&0.26&74&&0.37&76\\
B-DEIM-RRQR&0.54 &82&&0.18&71&&0.38&56\\
B-DEIM-MAxVol&0.55 &83&&0.18&71&&0.38&57\\
AdapBlock-RRQR&0.54&85&&0.18&79&&0.37&59\\
AdapBlock-MaxVol&0.54&86&&0.18&80&&0.38&59\\\hline
\end{tabular}
\end{table}

Table 4 presents the relative errors and runtimes, incorporating the SVD runtime, for the different algorithms. Despite the SVD runtime being the dominant factor in the overall runtime of the index selection process, it is evident that enhancing the efficiency of the index selection step contributes to an overall improvement. Specifically, in the comparison between DEIM and the block variants, while maintaining similar approximation qualities, the block DEIM algorithms exhibit at least a 10\% reduction in runtimes.}

\end{experiment}

\begin{experiment}{\rm Using two of the block DEIM algorithms proposed: the B-DEIM-MaxVol and B-DEIM-RRQR, we investigate how varying block sizes, i.e., $b=(2,\,5,\,10,\,20)$ may affect their approximation quality and computational efficiency.

Following the experiments in \cite{Voronin}, our test matrix in this experiment is a full-rank data set $A \in \R^{2000 \times 4000}$ that has the structure of the SVD, i.e., $A=U\Sigma V^T$. The matrices $U$ and $V$ have random orthonormal columns obtained via a QR factorization of a random Gaussian matrix, and the diagonal matrix $\Sigma$ has entries that are logspace ranging from $1$ to $10^{-3}$.  For each fixed block size, maintaining the properties of $A$, we generate five different test cases and compute the averages of the evaluation criteria for the range of $k$ values.

We observe in \cref{fig:3.6} that both algorithms become considerably faster for increasing block sizes. On the other hand, the approximation quality of the varying block sizes may not degrade significantly. In this experiment, given the various values of $k$, the errors are almost similar irrespective of the block size. 

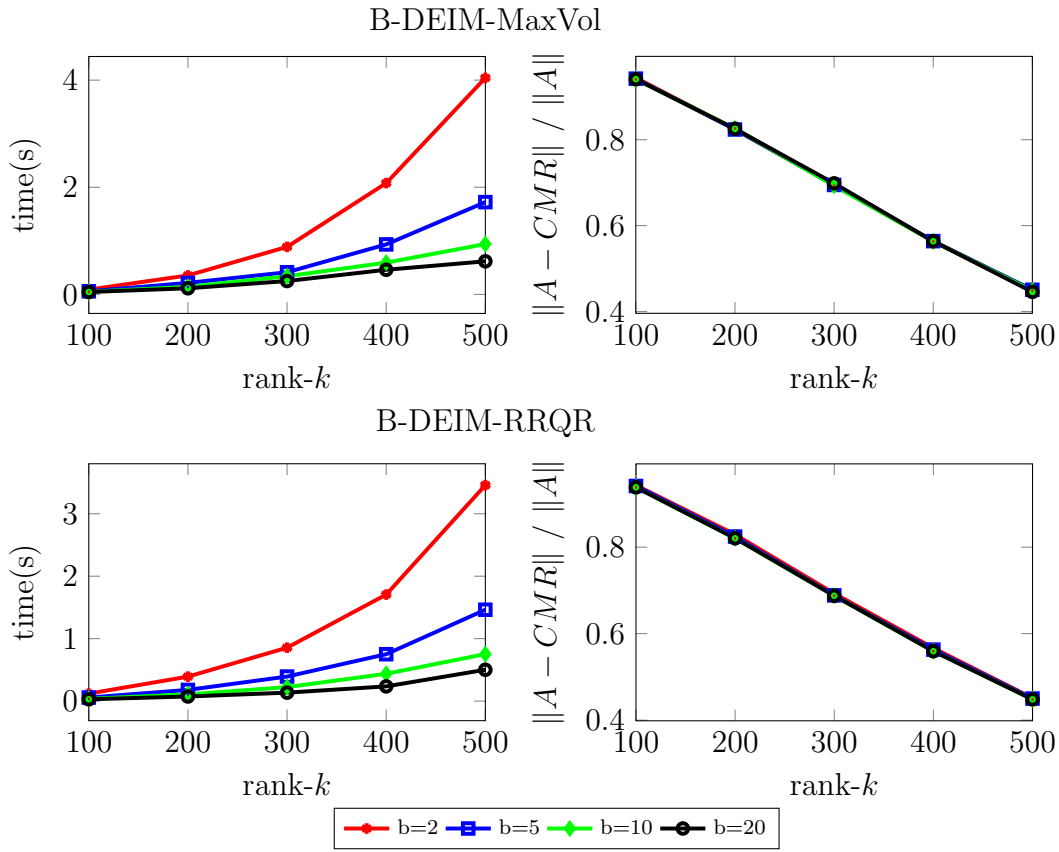
\begin{figure}[htb!]
\centering
{
%
%
\definecolor{mycolor1}{rgb}{0.00000,0.44700,0.74100}%
\definecolor{mycolor2}{rgb}{0.85000,0.32500,0.09800}%
\pgfplotsset{every axis title/.append style={at={(1,1)}}}
{\begin{tikzpicture}

\begin{groupplot}[
group style={
    group name=my plots,
    group size=2 by 2,
    xlabels at=edge bottom,
    ylabels at=edge left,
    horizontal sep=2cm,
    vertical sep=2cm,
    },
     xmin=100,
     xmax=500,
    legend style={at={(-0.2,-0.5)},anchor=south,legend  columns =4},
width=0.5\linewidth,
height=5cm
]

\nextgroupplot[title=B-DEIM-MaxVol,xlabel=rank-$k$, ylabel=time(s)]
\addplot [color=red,line width=1.5pt,mark=asterisk]
  table[row sep=crcr]{%
100	0.0887726\\
200	0.35407828\\
300	0.88673608\\
400	2.07701128\\
500	4.04147728\\
600	5.2506437\\
700	7.65894234\\
800	9.96627014\\
};

\addplot [color=blue,line width=1.5pt,mark=square]
  table[row sep=crcr]{%
100	0.0572176\\
200	0.21463516\\
300	0.41335524\\
400	0.93294458\\
500	1.72238766\\
600	2.41847948\\
700	3.24412402\\
800	4.31720902\\
};

\addplot [color=green,line width=1.5pt,mark=diamond]
  table[row sep=crcr]{%
100	0.0479965\\
200	0.14066016\\
300	0.33931272\\
400	0.59257522\\
500	0.93938296\\
600	1.34930446\\
700	1.89453\\
800	2.43790924\\
};

\addplot [color=black,line width=1.5pt,mark=o]
  table[row sep=crcr]{%
100	0.04416232\\
200	0.11384288\\
300	0.24922754\\
400	0.45949758\\
500	0.61713812\\
600	0.85792192\\
700	1.21624956\\
800	1.64648362\\
};

\nextgroupplot[xlabel=rank-$k$,ylabel={$\| A - CMR \|\ /\ \| A \|$}]

\addplot [color=red,line width=1.5pt,mark=asterisk]
  table[row sep=crcr]{%
100	0.944183630542758\\
200	0.824636631171269\\
300	0.694695567697066\\
400	0.563475768307043\\
500	0.449497461957857\\
600	0.350979370498818\\
700	0.274481342454704\\
800	0.207627564181034\\
};
\addplot [color=blue,line width=1.5pt,mark=square]
  table[row sep=crcr]{%
100	0.942116998418969\\
200	0.823758171666379\\
300	0.694763623467651\\
400	0.563955031307782\\
500	0.45069913463227\\
600	0.354130953116287\\
700	0.270022989545213\\
800	0.206547306530086\\
};
\addplot [color=green,line width=1.5pt,mark=diamond]
  table[row sep=crcr]{%
100	0.940566353955937\\
200	0.826143988492494\\
300	0.692802768991783\\
400	0.563295409052597\\
500	0.44882325007881\\
600	0.347239492200104\\
700	0.270953991263695\\
800	0.207147050140799\\
};
\addplot [color=black,line width=1.5pt,mark=o]
  table[row sep=crcr]{%
100	0.940878782315506\\
200	0.825408342540737\\
300	0.698449556266016\\
400	0.563754577551219\\
500	0.445781343033938\\
600	0.349067462478657\\
700	0.270372070191756\\
800	0.20477502453564\\
};
\nextgroupplot[title=B-DEIM-RRQR, xlabel=rank-$k$,ylabel= time(s)]

\addplot [color=red,line width=1.5pt,mark=asterisk]
  table[row sep=crcr]{%
100	0.11984622\\
200	0.39211038\\
300	0.85539432\\
400	1.705663\\
500	3.45741068\\
600	5.27365046\\
700	7.79719114\\
800	10.45609392\\
};

\addplot [color=blue,line width=1.5pt,mark=square]
  table[row sep=crcr]{%
100	0.05704802\\
200	0.17950766\\
300	0.39335772\\
400	0.7509832\\
500	1.46292196\\
600	2.16048092\\
700	3.21591964\\
800	4.51129194\\
};

\addplot [color=green,line width=1.5pt,mark=diamond]
  table[row sep=crcr]{%
100	0.0372964\\
200	0.1075827\\
300	0.2245618\\
400	0.43761452\\
500	0.75266596\\
600	1.20190492\\
700	1.66113032\\
800	2.27754234\\
};

\addplot [color=black,line width=1.5pt,mark=o]
  table[row sep=crcr]{%
100	0.0295753\\
200	0.0740089\\
300	0.13554286\\
400	0.23588602\\
500	0.50172998\\
600	0.59081848\\
700	0.99335288\\
800	1.25078242\\
};

\nextgroupplot[xlabel=rank-$k$, ylabel={$\| A - CMR \|\ /\ \| A \|$}]
\addplot [color=red,line width=1.5pt,mark=asterisk]
  table[row sep=crcr]{%
100	0.94316056613792\\
200	0.829292863521185\\
300	0.692958046061168\\
400	0.566782552134941\\
500	0.45113114604773\\
600	0.35259454789801\\
700	0.272149454379493\\
800	0.207962076962845\\
};
\addlegendentry{b=2}

\addplot [color=blue,line width=1.5pt,mark=square]
  table[row sep=crcr]{%
100	0.941526385992422\\
200	0.824549181382843\\
300	0.688590152192489\\
400	0.563073306402433\\
500	0.449982172307823\\
600	0.349350288570522\\
700	0.271919078323407\\
800	0.205493556465548\\
};
\addlegendentry{b=5}

\addplot [color=green,line width=1.5pt,mark=diamond]
  table[row sep=crcr]{%
100	0.938443403144106\\
200	0.820589013864922\\
300	0.687630102530473\\
400	0.559491161395726\\
500	0.448038117858879\\
600	0.347447750191122\\
700	0.269071158667418\\
800	0.206856717067825\\
};
\addlegendentry{b=10}

\addplot [color=black,line width=1.5pt,mark=o]
  table[row sep=crcr]{%
100	0.938443403144106\\
200	0.820589013864922\\
300	0.687630102530473\\
400	0.559491161395726\\
500	0.448038117858879\\
600	0.347447750191122\\
700	0.269071158667418\\
800	0.206856717067825\\
};
\addlegendentry{b=20}
\end{groupplot}

\end{tikzpicture}}%
}
\caption{Runtimes and average approximation errors for the B-DEIM-MaxVol (up) and B-DEIM-RRQR (down) CUR approximation algorithms as a function of $k$ for varying block sizes using large matrices of size $2000 \times 4000$.\label{fig:3.6}}
\end{figure}
}
\end{experiment}

\section{Conclusions}\label{sec:con}
This paper presents various block variants of the DEIM scheme for computing CUR decompositions. We exploit the advantages of the classical DEIM procedure, a column-pivoted QR decomposition, and the concept of maximum determinant or volume of submatrices to develop these block variants. We have then presented a version of the block DEIM, which allows for an adaptive choice of block size.

We perform the following procedures in the block DEIM based on RRQR; at each iteration step, we compute a QR factorization with column pivoting on the transpose of a block of singular vectors to obtain the indices corresponding to the first $b$ columns. Then, we update the next block of vectors using the interpolatory projection technique in the DEIM algorithm (repeat these two steps until all indices are selected). A similar procedure is used in the block DEIM based on {\sf MaxVol}; the difference here is instead of using a column-pivoted QR decomposition, we use the {\sf MaxVol} method. 

Numerical experiments illustrate that the accuracy of a CUR factorization using the newly proposed block DEIM procedures is comparable to the classical DEIM, {\sf MaxVol}, and QDEIM schemes. The experiments also demonstrate that the block variants, regarding computational speed, may have an advantage over the standard DEIM and {\sf MaxVol} algorithms. Relative to the QDEIM algorithm, the B-DEIM-RRQR scheme sometimes yields lesser approximation errors while maintaining comparable runtimes. Using the B-DEIM-RRQR and B-DEIM-MaxVol methods, we have also illustrated how increasing the block size improves the speed of the algorithms but may not necessarily degrade the approximation quality significantly.
\Cref{tab:overview} displays a schematic overview of some properties of the various methods. A Matlab code of the proposed algorithms is available via \href{https://github.com/perfectyayra/Block-discrete-empirical-interpolation-methods}{github.com/perfectyayra}.

\section*{Acknowledgements} 
This work has received funding from the European Union's Horizon 2020 research and innovation
programme under the Marie Sk\l odowska-Curie grant agreement No 812912.
We thank the editor and the two expert referees for their very valuable suggestions.

\end{document}